\documentclass{article}

\usepackage{amsmath,amsfonts,bm,amsthm}
\usepackage{amssymb,enumitem,mathtools,mathrsfs}
\usepackage[colorlinks=true, linkcolor=blue, citecolor=blue, urlcolor=blue]{hyperref}
\hypersetup{pdftitle={Graphs with three eigenvalues from group divisible designs},
  pdfauthor={Gary Greaves and Jose Yip}}
\usepackage{booktabs}
\usepackage{geometry}
\usepackage{graphicx}
\makeatletter
\newcommand*{\transpose}{%
  {\mathpalette\@transpose{}}%
}
\newcommand*{\@transpose}[2]{%
  \raisebox{\depth}{$\m@th#1\intercal$}%
}
\makeatother

\theoremstyle{plain}
\newtheorem{theorem}{Theorem}[section]
\newtheorem{proposition}[theorem]{Proposition}
\newtheorem{lemma}[theorem]{Lemma}
\newtheorem{corollary}[theorem]{Corollary}

\theoremstyle{definition}

\theoremstyle{remark}
\newtheorem{remark}[theorem]{Remark}

\newcommand{\FF}{\mathbb{F}}
\newcommand{\ZZ}{\mathbb{Z}}

\newcommand{\RR}{\mathbb{R}}
\newcommand{\Spec}{\operatorname{Spec}}
\newcommand{\crk}{\operatorname{cr}}
\newcommand{\tr}{\operatorname{tr}}
\newcommand{\rank}{\operatorname{rank}}

\makeatletter
\def\bbordermatrix#1{\begingroup \m@th
  \@tempdima 4.75\p@
  \setbox\z@\vbox{%
    \def\cr{\crcr\noalign{\kern2\p@\global\let\cr\endline}}%
    \ialign{$##$\hfil\kern2\p@\kern\@tempdima&\thinspace\hfil$##$\hfil
      &&\quad\hfil$##$\hfil\crcr
      \omit\strut\hfil\crcr\noalign{\kern-\baselineskip}%
      #1\crcr\omit\strut\cr}}%
  \setbox\tw@\vbox{\unvcopy\z@\global\setbox\@ne\lastbox}%
  \setbox\tw@\hbox{\unhbox\@ne\unskip\global\setbox\@ne\lastbox}%
  \setbox\tw@\hbox{$\kern\wd\@ne\kern-\@tempdima\left[\kern-\wd\@ne
    \global\setbox\@ne\vbox{\box\@ne\kern2\p@}%
    \vcenter{\kern-\ht\@ne\unvbox\z@\kern-\baselineskip}\,\right]$}%
  \null\;\vbox{\kern\ht\@ne\box\tw@}\endgroup}
\makeatother

\title{Graphs with three eigenvalues from group divisible designs}
\author{Gary Greaves\thanks{Division of Mathematical Sciences,
School of Physical and Mathematical Sciences, Nanyang Technological University,
21 Nanyang Link, Singapore 637371.
Email: \href{mailto:gary@ntu.edu.sg}{\texttt{gary@ntu.edu.sg}}.}
\and
Jose Yip\thanks{Division of Mathematical Sciences,
School of Physical and Mathematical Sciences, Nanyang Technological University,
21 Nanyang Link, Singapore 637371.
Email: \href{mailto:josezhengho.yip@ntu.edu.sg}{\texttt{josezhengho.yip@ntu.edu.sg}}.}}
\date{}

\begin{document}

\maketitle

\begin{abstract}
We construct biregular graphs with three distinct eigenvalues from semiregular group divisible designs.  
For every prime power $q$, we use the spectral structure of Muzychuk's Construction~A on $q+1$ copies of $\operatorname{AG}(2,q^2)$ to produce suitable group divisible designs.
The resulting biregular graphs have $2q^5+q^4+q^2$ vertices and distinct eigenvalues $q^5+q^3+q^2$, $q^2$, and $-q^2$.
Their coherent closures have rank $12$ and type $\left[\begin{smallmatrix}3&2\\2&5\end{smallmatrix}\right]$.  
We also determine the four-class symmetric association schemes that are the coherent closures of the auxiliary graphs in this family.
Computational searches provide further graphs on $81$ and $156$ vertices, using Muzychuk's Construction~B and a generalised quadrangle, respectively; their coherent closures have ranks $13$ and $10$.
The $156$-vertex example has a strongly regular design as its coherent closure, of type $\left[\begin{smallmatrix}3&2\\2&3\end{smallmatrix}\right]$, answering a question of the authors.
We also realise a $126$-vertex example by switching a strongly regular graph.
\end{abstract}

%% ============================================================================
\section{Introduction}
\label{sec:intro}

All graphs in this paper are finite, undirected, and simple.  
We write $\mathscr G_3(\theta_0,\theta_1,\theta_2)$ for the family of connected graphs with distinct eigenvalues $\theta_0>\theta_1>\theta_2$, and write $\mathscr G_3$ when the eigenvalues are not specified.  
Every graph in $\mathscr G_3$ has diameter $2$.  
The regular graphs in $\mathscr G_3$ are precisely the connected strongly regular graphs, but much less is known in the non-regular case.

Bridges and Mena~\cite{BridgesMena1981} studied cones with three distinct eigenvalues, while Muzychuk and Klin~\cite{MUZYCHUK1998191} approached the problem through coherent algebras.  
Muzychuk and Klin~\cite{MUZYCHUK1998191} and Van Dam~\cite[Section~2.2]{VANDAM1998101} also constructed nonregular graphs with three eigenvalues using switching.
Van Dam~\cite{VANDAM1998101} classified the graphs in $\mathscr G_3$ on at most $29$ vertices and those whose smallest eigenvalue is at least $-2$, and gave further constructions.
Further examples and classifications appear in~\cite{ddC99,chuang,Cheng_2016,Cheng18}.

There are, however, strong restrictions in the biregular case.
Van Dam~\cite{VANDAM1998101} showed that a graph in $\mathscr G_3$ with at most three distinct degrees has an equitable degree partition; see Theorem~\ref{thm:equitableValency}.  
Cheng, Gavrilyuk, Greaves and Koolen~\cite{Cheng_2016} proved further arithmetic restrictions on the eigenvalues and quotient matrix; see Theorem~\ref{thm:bireg}.  
With respect to the degree partition, the adjacency matrix of a biregular graph has the form $\left [\begin{smallmatrix}A_1&B\\B^\transpose&A_2\end{smallmatrix} \right ]$.
After interchanging the degree classes if necessary, suppose that the class corresponding to $A_1$ induces $\overline{mK_n}$ or $mK_n$, with $m,n\geqslant2$.  
Cheng et al.~\cite[Theorem~5.1]{Cheng_2016} showed that $B$ is then the incidence matrix of a group divisible design; see Theorem~\ref{thm:2-des}.
Its intersection parameters are determined by the spectrum of the graph and the quotient matrix of the degree partition.  
If the design is semiregular in the sense of Bose and Connor~\cite{BoseConnor1952}, then $n=-\theta_2$ in the complete multipartite case and $n=\theta_1+1$ in the disjoint-union-of-cliques case.  
This result may be viewed as a degree-partitioned analogue of Higman's strongly regular designs~\cite{Higman1988}.
Related design constructions also occur in Rowlinson's work \cite{rowlinson}, which characterises equality cases in a minimum-degree bound using strongly regular graphs and symmetric or quasi-symmetric designs.
Greaves and Yip~\cite{GreavesYip2024} characterise the graphs with three eigenvalues and coherent rank $9$ in terms of quasi-symmetric designs.

For semiregular designs in the complete multipartite case, we prove a converse.  
Let $\mathcal D$ be a semiregular group divisible design with at least two groups and incidence matrix $B$.  
If a regular graph $\Gamma$ on the blocks of $\mathcal D$ satisfies two matrix identities involving $BA(\Gamma)$ and $A(\Gamma)^2+B^\transpose B$, together with a scalar compatibility condition, then $\mathcal D$ and $\Gamma$ determine a graph in $\mathscr G_3$; see
Theorem~\ref{thm:gdd-to-3ev}.  
We call $\Gamma$ an \emph{auxiliary graph} of $\mathcal D$.
The design parameters allow at most two possible degrees for an auxiliary graph and determine its spectrum, which is contained in $\{k_\Gamma,\theta_1+\theta_2,\theta_1,\theta_2\}$; see Lemma~\ref{lem:auxSpec}.

These restrictions make a parameter search practical.
Table~\ref{tab:feasible-parameters} contains the twelve parameter sets with $v\leqslant100$ and $\lambda_2\leqslant50$ that satisfy the required divisibility, integrality, and spectral conditions.
Five have $\lambda_1=0$ and hence arise from transversal designs, or equivalently orthogonal arrays of strength $2$.  
Three of these have empty auxiliary graph and recover examples of Van Dam~\cite{VANDAM1998101}.

The third row is realised by switching a strongly regular graph with parameters $(126,45,12,18)$ with respect to an induced $K_{9,9}$.
The resulting biregular graph has degrees $81$ and $51$ and spectrum $\{[57]^1,[3]^{89},[-9]^{36}\}$.  
Its complete multipartite degree class yields the corresponding GDD by Theorem~\ref{thm:2-des}; see Remark~\ref{rem:switch126}.

Our main result is a construction of an infinite family of graphs $\Gamma_q\in\mathscr G_3(q^5+q^3+q^2,q^2,-q^2)$ for each prime power $q$. 
The graph $\Gamma_q$ has $2q^5+q^4+q^2$ vertices, and the case $q=2$ realises the fifth row of Table~\ref{tab:feasible-parameters}; see Section~\ref{sec:design}.
For each $q$, we use the specialisation of Muzychuk's Construction~A~\cite{Muzychuk2007} on $q+1$ copies of $\operatorname{AG}(2,q^2)$ as the auxiliary graph.
The used parallel classes in each copy, their assignments to the other copies, and the line-matching bijections may all be chosen arbitrarily; every graph in this specialisation admits a compatible semiregular GDD.
Muzychuk's construction supplies the regular auxiliary graphs.
Our additional step constructs compatible incidence matrices uniformly over all these choices and verifies the mixed, quadratic, and scalar conditions of Theorem~\ref{thm:gdd-to-3ev}.
The resulting design has parameters
\[
(v,b;r,k;\lambda_1,\lambda_2;m,n)=\bigl(q^2(q^3+1),q^4(q+1);\ q^3(q+1),q(q^3+1);\
 q^3,q^2(q+1);\ q^3+1,q^2\bigr).
\]

The fourth row is realised using Muzychuk's Construction~B~\cite{Muzychuk2007} with six copies of $\operatorname{AG}(2,3)$ and base graph $K_{3,3}$.
A clique of compatible vectors in its $3$-eigenspace supplies a semiregular GDD with parameters $(27,54;18,9;0,6;9,3)$ and hence a biregular graph on $81$ vertices with spectrum $\{[33]^1,[6]^{23},[-3]^{57}\}$; see Theorem~\ref{thm:row4-muzychuk}.
Its nonprincipal quotient eigenvalue is $\theta_1$, whereas the infinite family has $\theta_2$.

For the seventh row, we use the point graph of a generalised quadrangle of order $(5,3)$, which is strongly regular with parameters $(96,20,4,4)$.
A search for successive incidence rows using binary linear constraints produces a compatible semiregular GDD with parameters $(60,96;48,30;16,24;15,4)$.
This yields a biregular graph on $156$ vertices with spectrum $\{[80]^1,[8]^{45},[-4]^{110}\}$; see Theorem~\ref{thm:quadrangle-graph}.

We also determine the association schemes and coherent closures of these constructions.  
The five relations on the block set of the infinite family form a
four-class symmetric association scheme.  
We determine its first eigenmatrix and show that its Bose--Mesner algebra is the coherent closure of the auxiliary graph.  
We also determine the type of the coherent closure of the biregular graphs $\Gamma_q$; see Section~\ref{sub:coherent-types}.
The coherent closure of the $156$-vertex graph is a strongly regular design in the sense of Higman, with type $\left[\begin{smallmatrix}3&2\\2&3\end{smallmatrix}\right]$.
This answers \cite[Question~8.5]{GreavesYip2024} in the affirmative.

The paper is organised as follows.

In the remainder of this section, we collect preliminary results and introduce notation.
In Section~\ref{sec:gdds}, we prove relations between biregular graphs with three distinct eigenvalues and GDDs.  
In Section~\ref{sec:multiplicative}, we derive the spectrum of the auxiliary graph and the GDD parameters in the case the corresponding biregular graph is multiplicative.
In Section~\ref{sec:design}, we provide an infinite family of multiplicative biregular graphs with three distinct eigenvalues.
In Sections~\ref{sec:row4-muzychuk} and~\ref{sec:quadrangle}, we provide the two further constructions using a computational search. 
Finally, in Section~\ref{sec:cr} we determine the emergent association schemes and coherent closures.

\subsection{Notation and preliminary results}
\label{sub:notation}\label{sec:3ev}

For background in algebraic graph theory, see Godsil and Royle~\cite{God01}.
We use $I_n$, $J_n$, $O_n$ and $\mathbf1_n$ for the identity matrix, all-ones matrix, zero matrix and all-ones column vector, respectively, omitting subscripts when their sizes are clear.
The symbols $\otimes$ and $\circ$ denote the Kronecker and entrywise products.  
In a spectrum, $[\theta]^m$ denotes an eigenvalue $\theta$ of multiplicity $m$.
For a graph $H=(\mathfrak X,E)$, we write $A(H)$ for its adjacency matrix, $\overline H$ for its complement and $H[U]$ for its induced subgraph on $U$.
We write $K_n$ for the complete graph and $mK_n$ for the disjoint union of $m$ copies of $K_n$; thus $\overline{mK_n}$ is complete multipartite with $m$ parts of size $n$.  
An empty graph has no edges.

The \textbf{adjacency algebra} of $H$ consists of the polynomials in $A(H)$, and its \textbf{Seidel matrix} is $S(H)=J-I-2A(H)$.
We write $\operatorname{PG}(d,q)$ and $\operatorname{AG}(d,q)$ for the projective and affine spaces of dimension $d$ over $\FF_q$, and $\langle x\rangle$ for the one-dimensional span of a nonzero vector $x$.
A graph is \textbf{regular}, \textbf{biregular} or \textbf{triregular} if it has one, two or three distinct degrees, respectively.

By the Perron--Frobenius theorem \cite[Theorem~2.2.1]{BrouwerHaemers2012}, the spectrum of $\Gamma\in\mathscr G_3(\theta_0,\theta_1,\theta_2)$ has the form
\[
    \Spec(A(\Gamma))=\left\{[\theta_0]^1,[\theta_1]^{m_1},[\theta_2]^{m_2}\right\},
\]
where $m_1+m_2=|\mathfrak X|-1$.

\begin{lemma}[{\cite[Section~2]{Cheng_2016}}]\label{lem:rank-one}
Let $A$ be the adjacency matrix of a graph in $\mathscr G_3(\theta_0,\theta_1,\theta_2)$.  
There is a positive $\theta_0$-eigenvector $\bm\alpha$ such that
\begin{equation}\label{eqn:3ev}
    A^2-(\theta_1+\theta_2)A+\theta_1\theta_2I
       =\bm\alpha\bm\alpha^\transpose.
\end{equation}
For each vertex $\mathsf x$ of degree $d_{\mathsf x}$, we have $\alpha_{\mathsf x}^2=d_{\mathsf x}+\theta_1\theta_2$.
Moreover, $\theta_1\geqslant0$ and  $\theta_2\leqslant-\sqrt2$.
\end{lemma}

A \textbf{strongly regular graph} with parameters $(v,k,a,c)$ is a $k$-regular graph on $v$ vertices, where $0<k<v-1$, in which adjacent vertices have $a$ common neighbours and non-adjacent vertices have $c$ common neighbours.  
Its adjacency matrix satisfies $A^2=kI+aA+c(J-I-A)$.

Let $\pi=\{\pi_1,\dots,\pi_t\}$ be a partition of $V(\Gamma)$.
We call $\pi$ \textbf{equitable} if there are constants $n_{ij}$ such that every vertex in $\pi_i$ has $n_{ij}$ neighbours in $\pi_j$.  
The matrix $Q_\pi(\Gamma)=(n_{ij})$ is the \textbf{quotient matrix}.  
If $\pi$ is equitable, then $\Spec(Q_\pi(\Gamma))\subseteq\Spec(A(\Gamma))$; see \cite[Lemma~2.3.1]{BrouwerHaemers2012}.

The \textbf{degree partition} $\mathfrak D(\Gamma)$ has as its parts the sets of vertices of the same degree.  
We use the following result of Van Dam.

\begin{theorem}[{\cite[Section 4]{VANDAM1998101}}]
\label{thm:equitableValency}
    Let $\Gamma \in \mathscr{G}_3(\theta_0,\theta_1,\theta_2)$ have at most three distinct degrees. Then the degree partition $\mathfrak{D}(\Gamma)$ is equitable, and $\theta_0$ is an eigenvalue of the quotient matrix $Q_{\mathfrak{D}(\Gamma)}(\Gamma)$.
\end{theorem}

Throughout, when $\Gamma \in \mathscr G_3(\theta_0,\theta_1,\theta_2)$ is biregular we write $\{V_1,V_2\}$ for its degree partition, $n_i \coloneqq |V_i|$, $k_i$ for the common degree of a vertex of $V_i$, and
\[
    Q_{\mathfrak D(\Gamma)}(\Gamma) = \begin{bmatrix} k_{11} & k_{12} \\ k_{21} & k_{22}\end{bmatrix}
\]
for the quotient matrix, so that $k_i = k_{i1}+k_{i2}$.  
By Lemma~\ref{lem:rank-one}, the Perron eigenvector $\bm\alpha$ of \eqref{eqn:3ev} is constant on each $V_i$, with value $\alpha_i = \sqrt{k_i + \theta_1\theta_2}$.

For non-bipartite biregular graphs, we require the following stronger restrictions.

\begin{theorem}[{\cite[Theorem 4.3]{Cheng_2016}}]
\label{thm:bireg}
    Let $\Gamma \in \mathscr{G}_3(\theta_0,\theta_1,\theta_2)$ be a non-bipartite biregular graph with degrees $k_1$ and $k_2$, and let $Q=Q_{\mathfrak D(\Gamma)}(\Gamma)$.  
    Then
	\begin{enumerate}
		\item[\textup{(i)}] All eigenvalues of $\Gamma$ are integers.
        \item[\textup{(ii)}] The quotient matrix $Q$ has eigenvalues $\theta_0$ and $\theta$, where $\theta \in \{\theta_1,\theta_2\}$.
		\item[\textup{(iii)}] Letting $\{\theta, \theta'\} = \{\theta_1, \theta_2\}$, we have $\alpha_1\alpha_2 = -\theta(\theta' +1)$.
        \item[\textup{(iv)}] If either part of the degree partition induces an empty subgraph, then $\theta = \theta_2$.
  \end{enumerate}
\end{theorem}

%% ============================================================================
\section{Semiregular group divisible designs}
\label{sec:gdds}

\subsection{Group divisible designs}
\label{sec:gdd}

A set $\mathfrak X$ of cardinality $v$ together with a collection of $k$-subsets $\mathfrak B_1,\dots,\mathfrak B_b$ (called blocks) of $\mathfrak X$ is called a $(v,b,r,k)$-\textbf{configuration} if each point of $\mathfrak X$ occurs in precisely $r$ blocks.
A standard double counting argument reveals that $vr = bk$.

A $(v,b,r,k)$-configuration $\mathcal D$ is called a \textbf{group divisible design} (GDD) if its $v$ points can be partitioned into $m$ sets $\mathfrak T_1, \dots, \mathfrak T_m$, each of size $n \geqslant 2$, such that any two points $x \in \mathfrak T_i$ and $y \in \mathfrak T_j$ occur together in $\lambda_1$ blocks if $i=j$ and $\lambda_2$ blocks otherwise.
We say that $\mathcal D$ is a GDD with parameters $(v,b;r,k;\lambda_1,\lambda_2;m,n)$. 

Clearly, $v = mn$.  
Furthermore, using a standard double counting argument for the number of points $x \in \mathfrak X$ that appear in the same block as a fixed point $y \in \mathfrak X$ yields the identity
\begin{equation}
\label{eqn:gdd}
    r(k-1)=(n-1)\lambda_1+n(m-1)\lambda_2.
\end{equation}

The \textbf{incidence matrix} $B$ of $\mathcal D$ is defined as the $v \times b$ matrix $B$ where
\[
B_{x,k} = \begin{cases}
    1, & \text{ if } x \in \mathfrak B_k; \\
    0, & \text{ otherwise.}
\end{cases}
\]
Clearly, the rows and columns of an incidence matrix $B$ of $\mathcal D$ can be ordered such that
\begin{equation}
    \label{eqn:inc}
    B B^\transpose = rI + \lambda_1 I_m\otimes(J_n-I_n) + \lambda_2(J_m-I_m)\otimes J_n.
\end{equation}
It follows from \eqref{eqn:inc} and \eqref{eqn:gdd} that 
\begin{equation*}
    \Spec(B B^{\transpose}) = \left \{[r k ]^1,\ [rk - v\lambda_2]^{m-1},\ [r-\lambda_1]^{m(n-1)} \right \}.
\end{equation*}
Conversely, if a $v \times b$ $\{0,1\}$-matrix $B$ satisfies \eqref{eqn:inc} for some $\lambda_1, \lambda_2, m,n \in \mathbb Z$ and $B^\transpose \mathbf 1 = k \mathbf 1$ then we can deduce that $B$ is the incidence matrix of a GDD with parameters $(v,b;r,k;\lambda_1,\lambda_2;m,n)$.

Following Bose--Connor~\cite{BoseConnor1952}, a GDD is \textbf{singular} if $r=\lambda_1$, \textbf{semiregular} if $r>\lambda_1$ and $rk=v\lambda_2$, and \textbf{regular} if $r>\lambda_1$ and $rk>v\lambda_2$.
By \cite[Lemma 2.9.2]{BrouwerHaemers2012}, the matrices $B^\transpose B$ and $B B^\transpose$ share the same non-zero spectrum.
Hence, if $B$ is the incidence matrix of a semiregular GDD, then
\begin{equation}
\label{eqn:specbtb}
    \Spec(B^{\transpose}B) = \left \{[r k ]^1,\ [r-\lambda_1]^{m(n-1)}, [0]^{b-m(n-1)-1} \right \}.
\end{equation}

\subsection{Biregular graphs to GDDs}

We begin by extending a result of Cheng et al.~\cite[Theorem~5.1]{Cheng_2016} to the case of semiregular GDDs.  

\begin{theorem}\label{thm:2-des}
Let $\Gamma\in\mathscr G_3(\theta_0,\theta_1,\theta_2)$ be connected and biregular, with degree partition $\{V_1,V_2\}$, adjacency matrix $\left[\begin{smallmatrix}A_1&B\\B^\transpose&A_2\end{smallmatrix}\right]$, and degree-partition quotient matrix $\left[\begin{smallmatrix}k_{11}&k_{12}\\k_{21}&k_{22}\end{smallmatrix}\right]$.
Suppose that $n_1=mn$ for integers $m,n\geqslant2$, and that $\Gamma[V_1]=\overline{mK_n}$ or $\Gamma[V_1]=mK_n$.
Then $B$ is the incidence matrix of a GDD $\mathcal D$ with parameters $(n_1,n_2;k_{12},k_{21};\lambda_1,\lambda_2;m,n)$, where
\begin{equation}
\label{eqn:l1l2}
    (\lambda_1,\lambda_2)=
\begin{cases}
\bigl(k_{12}+\theta_1\theta_2,
      k_{12}+(\theta_1+1)(\theta_2+1)-1+n\bigr),
   &\text{if $\Gamma[V_1]=\overline{mK_n}$},\\[3pt]
\bigl(k_{12}+(\theta_1+1)(\theta_2+1),
      k_{11}+k_{12}+\theta_1\theta_2\bigr),
   &\text{if $\Gamma[V_1]=mK_n$}.
\end{cases}
\end{equation}
If $\mathcal D$ is semiregular, then both of the following conclusions hold:
\[
n=
\begin{cases}
-\theta_2,&\text{if $\Gamma[V_1]=\overline{mK_n}$},\\
\theta_1+1,&\text{if $\Gamma[V_1]=mK_n$},
\end{cases}
\]
and
\[
\Spec(A_2)\subseteq
\begin{cases}
\{k_{22},\theta_1,-n,\theta_1-n\},
    &\text{if $\Gamma[V_1]=\overline{mK_n}$},\\
\{k_{22},n-1,\theta_2,n+\theta_2\},
    &\text{if $\Gamma[V_1]=mK_n$}.
\end{cases}
\]
\end{theorem}

\begin{proof}
By Theorem~\ref{thm:equitableValency} the degree-partition is equitable.
Hence, $B\mathbf1_{n_2}=k_{12}\mathbf1_{n_1}$ and $B^\transpose\mathbf1_{n_1}=k_{21}\mathbf1_{n_2}$.
The $(1,1)$-block of \eqref{eqn:3ev} is
\begin{equation}\label{eq:BBstar}
BB^\transpose
   =-A_1^2+(\theta_1+\theta_2) A_1-\theta_1\theta_2I+\alpha_1^2J,
\end{equation}
where $\alpha_1^2=k_{11}+k_{12}+\theta_1\theta_2$.
In the complete multipartite case, $A_1=J-I_m\otimes J_n$ and $A_1^2=(n_1-2n)J+n(I_m\otimes J_n)$.
In the disjoint-union-of-cliques case, $A_1=I_m\otimes J_n-I$ and $A_1^2=I+(n-2)(I_m\otimes J_n)$.
Substitution in \eqref{eq:BBstar} yields
\[
BB^\transpose
   =k_{12}I+\lambda_1(I_m\otimes J_n-I)+\lambda_2(J-I_m\otimes J_n),
\]
with the values of $\lambda_1$ and $\lambda_2$ given by \eqref{eqn:l1l2}.
Since $B$ is a $\{0,1\}$-matrix with constant row and column sums, this proves the GDD assertion.

For the rest of the proof, assume that $\mathcal D$ is semiregular.
On the $(m-1)$-dimensional space of vectors that are constant on each group and orthogonal to $\mathbf1_{n_1}$, the eigenvalue of $BB^\transpose$ is
\begin{equation}
\label{eqn:beta}
    \beta=k_{12}+(n-1)\lambda_1-n\lambda_2.
\end{equation}
Using semiregularity and \eqref{eqn:gdd}, we have $\beta=k_{12}k_{21}-n_1\lambda_2=0$.
On the other hand, substituting \eqref{eqn:l1l2} into \eqref{eqn:beta} yields
\[
\beta=
\begin{cases}
-(n+\theta_1)(n+\theta_2),
    &\text{if $\Gamma[V_1]=\overline{mK_n}$},\\
-\bigl(n-1-\theta_1\bigr)\bigl(n-1-\theta_2\bigr),
    &\text{if $\Gamma[V_1]=mK_n$}.
\end{cases}
\]
By Lemma~\ref{lem:rank-one}, we have $\theta_1\geqslant0$ and $\theta_2<0$.
Furthermore, since $n\geqslant2$, the factors $n+\theta_1$ and $n-1-\theta_2$ are positive.
Consequently, $n=-\theta_2$ in the first case and $n=\theta_1+1$ in the second.

Since $n_1k_{12}=n_2k_{21}$, we have
\[
    \bigl(B-\frac{k_{21}}{n_1}J_{n_1,n_2}\bigr)
        \mathbf1_{n_2}=0,
    \qquad
    \bigl(B-\frac{k_{21}}{n_1}J_{n_1,n_2}\bigr)^\transpose
        \mathbf1_{n_1}=0.
\]
The vanishing of $\beta$ shows that $B^\transpose$ annihilates every vector orthogonal to $\mathbf1_{n_1}$ that is constant on each group.
Therefore, the sum of the incidence rows in any group is independent of the group, whence $(I_m \otimes J_n)B=(k_{21}/m)J$ and hence $(I_m \otimes J_n)\bigl(B-\frac{k_{21}}{n_1}J_{n_1,n_2}\bigr)=0$.
Thus,
\begin{equation}
    \label{eqn:A1J}
    A_1\bigl(B-\frac{k_{21}}{n_1}J_{n_1,n_2}\bigr)
=a\bigl(B-\frac{k_{21}}{n_1}J_{n_1,n_2}\bigr),
\end{equation}
where $a=0$ in the complete multipartite case and $a=-1$ in the disjoint-union-of-cliques case.
The $(1,2)$-block of \eqref{eqn:3ev} is
\[
A_1B+BA_2-(\theta_1+\theta_2) B=\alpha_1\alpha_2J.
\]
Using \eqref{eqn:A1J}, we obtain
\[
\bigl(B-\frac{k_{21}}{n_1}J_{n_1,n_2}\bigr)A_2
=(\theta_1+\theta_2-a)\bigl(B-\frac{k_{21}}{n_1}J_{n_1,n_2}\bigr),
\]
and transposing yields,
\[
A_2\bigl(B-\frac{k_{21}}{n_1}J_{n_1,n_2}\bigr)^\transpose
=(\theta_1+\theta_2-a)\bigl(B-\frac{k_{21}}{n_1}J_{n_1,n_2}\bigr)^\transpose.
\]
Thus $A_2$ acts as the scalar $\theta_1+\theta_2-a$ on $\operatorname{Im}\bigl((B-\frac{k_{21}}{n_1}J_{n_1,n_2})^\transpose\bigr)$.
Moreover,
\[
\mathbb R^{n_2}
 =\langle\mathbf1_{n_2}\rangle
   \oplus\operatorname{Im}\bigl((B-\frac{k_{21}}{n_1}J_{n_1,n_2})^\transpose\bigr)
   \oplus\ker B
\]
is an orthogonal decomposition.
Indeed, $k_{21}>0$ by connectedness, and the row space of $B$ is $\langle\mathbf1_{n_2}\rangle\oplus \operatorname{Im}\bigl((B-\frac{k_{21}}{n_1}J_{n_1,n_2})^\transpose\bigr)$.
These subspaces are invariant under the symmetric matrix $A_2$.
Since $\ker B\subseteq\mathbf1_{n_2}^{\perp}$ the $(2,2)$-block of \eqref{eqn:3ev} restricted to $\ker B$ yields
\[
(A_2-\theta_1I)(A_2-\theta_2I)=0.
\]
The eigenvalues of $A_2$ therefore belong to $\{k_{22},\theta_1+\theta_2-a,\theta_1,\theta_2\}$.
\end{proof}

\subsection{GDDs to biregular graphs}

Theorem~\ref{thm:2-des} shows that a biregular graph with three eigenvalues yields a GDD when one degree class induces a complete multipartite graph with at least two parts of size at least two.
We now give sufficient conditions in the reverse direction, constructing a graph in $\mathscr G_3$ from a semiregular GDD, a regular graph on its blocks, and a scalar compatibility condition. 

Define the matrix $X_{\mathcal D}$ to have rows and columns indexed by $\mathfrak X$, where $i$ is adjacent to $j$ if and only if $i$ and $j$ are in different groups.

\begin{remark}
\label{rem:semireg}
    Semiregular GDDs have the property that every block intersects every group in the same number of points.
    Since $\mathcal{D}$ is a semiregular GDD, every block intersects every group in exactly $k/m$ points. 
    This implies $(I_m \otimes J_n) B = \frac{k}{m} J_{v,b}$. 
Thus, $X_{\mathcal D} B = (J_v - I_m \otimes J_n) B = k J_{v,b} - \frac{k}{m} J_{v,b} = \frac{k(m-1)}{m} J_{v,b}$. 
\end{remark}

\begin{theorem}\label{thm:gdd-to-3ev}
Let $B$ be the incidence matrix of a semiregular group divisible design $\mathcal D$ with parameters $(v,b;r,k;\lambda_1,\lambda_2;m,n)$, where $m,n\geqslant2$.
Suppose there is a $k_\Gamma$-regular graph $\Gamma$ on $b$ vertices such that
\begin{align}
    B A(\Gamma) &= (\lambda_2-\lambda_1-n)B
        + \frac{k}{v}(k_\Gamma+\lambda_1-\lambda_2+n)J,
        \label{eqn:aux-mixed}\\[3pt]
         A(\Gamma)^2-(\lambda_2-\lambda_1-n)A(\Gamma)+B^\transpose B
        &=\frac{1}{b}\big((k_\Gamma+\lambda_1-\lambda_2)(k_\Gamma+n)+rk\big)J
                     +n(\lambda_2-\lambda_1)I,
        \label{eqn:aux-block}\\[3pt]
        \left(\frac{k(m-1)}{m}
                 +\frac{k}{v}(k_\Gamma+\lambda_1-\lambda_2+n)\right)^2
        &=\frac{v-n+\lambda_1}{b}
                     \big((k_\Gamma+\lambda_1-\lambda_2)(k_\Gamma+n)+rk\big).
        \label{eqn:aux-scalar}
\end{align}
Then the block matrix $\left [ \begin{smallmatrix} X_{\mathcal D} & B \\ B^\transpose & A(\Gamma) \end{smallmatrix} \right ]$ is the adjacency matrix of a connected graph in
\[
\mathscr{G}_3 \left (\frac{v - n + k_\Gamma + \sqrt{(v - n - k_\Gamma)^2 + 4rk}}{2}, \lambda_2 - \lambda_1, -n \right ),
\]
which is biregular if and only if $v-n+r\ne k_\Gamma+k$; otherwise, it is regular.
\end{theorem}

\begin{proof}
Since $\mathcal D$ is semiregular, \eqref{eqn:gdd} becomes $r-\lambda_1=n(\lambda_2-\lambda_1)>0$.
In particular, $\lambda_2-\lambda_1>0$ and $k>0$.
We can rearrange \eqref{eqn:aux-mixed} as
\[
    \left(B-\frac{k}{v}J_{v,b} \right)A(\Gamma)= \left(\lambda_2-\lambda_1-n \right) \left( B-\frac{k}{v}J_{v,b} \right).
\]
Since $(B-\frac{k}{v}J_{v,b})$ is nonzero, $\lambda_2-\lambda_1-n$ is an eigenvalue of $A(\Gamma)$, and $\lambda_2-\lambda_1-n \leqslant k_\Gamma$.
Therefore,
\[
    C\coloneqq\frac{k(m-1)}{m}
          +\frac{k}{v}(k_\Gamma-\lambda_2+\lambda_1+n)
      =\frac{k}{v}(v+k_\Gamma-\lambda_2+\lambda_1)
      \geqslant\frac{k}{v}(v-n)>0.
\]
Since $v-n+\lambda_1>0$, we apply \eqref{eqn:aux-scalar} to deduce $d\coloneqq\frac{1}{b}
       \big((k_\Gamma-\lambda_2+\lambda_1)(k_\Gamma+n)+rk\big)>0$.
We may therefore set
\[
    \alpha_1=\sqrt{v-n+\lambda_1},\qquad
    \alpha_2=\sqrt d
\]
and \eqref{eqn:aux-scalar} implies $\alpha_1\alpha_2=C$.
Let $A=\left [\begin{smallmatrix}X_{\mathcal D}&B\\
                       B^\transpose&A(\Gamma)\end{smallmatrix}\right ]$.
We claim that
\begin{equation}\label{eqn:1p}
    A^2-(\lambda_2-\lambda_1-n) A-n(\lambda_2-\lambda_1) I
       =\begin{bmatrix}
          \alpha_1^2J_v&\alpha_1\alpha_2J_{v,b}\\
          \alpha_1\alpha_2J_{b,v}&\alpha_2^2J_b
        \end{bmatrix}.
\end{equation}

We have $X_{\mathcal D}=J_v-I_m\otimes J_n$, so $X_{\mathcal D}^2=(v-2n)J_v+nI_m\otimes J_n$ and
\[
BB^\transpose=(r-\lambda_1)I_v-(\lambda_2-\lambda_1) I_m\otimes J_n+\lambda_2J_v.
\]
Using $r-\lambda_1=n(\lambda_2-\lambda_1)$, the $(1,1)$-block of the left-hand side of \eqref{eqn:1p} is therefore $(v-n+\lambda_1)J_v=\alpha_1^2J_v$.
By Remark~\ref{rem:semireg} and \eqref{eqn:aux-mixed}, the $(1,2)$-block of left-hand side of \eqref{eqn:1p} is
\[
    X_{\mathcal D}B+BA(\Gamma)-(\lambda_2-\lambda_1-n) B=CJ_{v,b}
       =\alpha_1\alpha_2J_{v,b}.
\]
The $(2,1)$-block follows by transposition, and the $(2,2)$-block is exactly \eqref{eqn:aux-block}.  
This proves \eqref{eqn:1p}.

The complete multipartite graph on the point vertices is connected, since $m\geqslant2$, and every block vertex has $k>0$ point neighbours.
Hence the graph with adjacency matrix $A$ is connected.
Set $\bm\alpha^\transpose = [\alpha_1\mathbf1_v^\transpose,\alpha_2\mathbf1_b^\transpose]$.
Then \eqref{eqn:1p} becomes
\[
    (A-(\lambda_2-\lambda_1) I)(A+nI)=\bm\alpha\bm\alpha^\transpose.
\]
Since $A$ commutes with the right-hand side, it preserves $\langle\bm\alpha\rangle$.
The Perron--Frobenius theorem~\cite[Theorem 2.2.1]{BrouwerHaemers2012} therefore identifies $\bm\alpha$ as a Perron eigenvector, with simple eigenvalue $\theta_0$.
On $\bm\alpha^\perp$, we have $(A-(\lambda_2-\lambda_1) I)(A+nI) = O$, so all remaining eigenvalues belong to $\{\lambda_2-\lambda_1,-n\}$.
Both must occur, otherwise $A$ would have a quadratic minimal polynomial, forcing $(A^2)_{xy}=0$ for every pair of distinct nonadjacent vertices $x,y$.  
However, two distinct point vertices in the same group are nonadjacent and have at least $v-n>0$ common neighbours.
Moreover,
\[
    (\theta_0-\lambda_2+\lambda_1)(\theta_0+n)=\|\bm\alpha\|^2>0,
\]
and $\theta_0>0$, so $\theta_0>\lambda_2-\lambda_1$.

The point--block partition is equitable, with quotient matrix
\[
    Q=\begin{bmatrix}v-n&r\\k&k_\Gamma\end{bmatrix}.
\]
Its positive eigenvector $(\alpha_1,\alpha_2)^\transpose$ has eigenvalue $\theta_0$, which is therefore the larger root of its characteristic polynomial.  
Thus
\[
    \theta_0=
       \frac{v-n+k_\Gamma+
          \sqrt{(v-n-k_\Gamma)^2+4rk}}{2}.
\]
The row sums of $Q$ are $v-n+r$ and $k_\Gamma+k$, respectively.
These give the two vertex degrees, proving the final assertion.
\end{proof}

In view of Theorem~\ref{thm:gdd-to-3ev}, given a GDD $\mathcal D$, we call a graph $\Gamma$ that satisfies the hypotheses of Theorem~\ref{thm:gdd-to-3ev} an \textbf{auxiliary graph} of $\mathcal D$.

\begin{remark}\label{rem:degree-compatibility}
Under the hypotheses of Theorem~\ref{thm:gdd-to-3ev}, set $\theta_1=\lambda_2-\lambda_1$ and $\theta_2=-n$.
The diagonal entries of \eqref{eqn:aux-block}, together with \eqref{eqn:aux-scalar}, imply
\begin{equation}\label{eqn:degree-compatibility}
    \frac{k^2}{v^2}(v+k_\Gamma-\theta_1)^2
       =(v-n+\lambda_1)(k_\Gamma+k+\theta_1\theta_2).
\end{equation}
The point--block partition is equitable, with quotient matrix $Q=\left [ \begin{smallmatrix}v-n&r\\k&k_\Gamma\end{smallmatrix}\right ]$.
By Theorem~\ref{thm:gdd-to-3ev} and the spectral inclusion for equitable quotients, its nonprincipal eigenvalue is $\theta\in\{\theta_1,\theta_2\}$, so $(v-n-\theta)(k_\Gamma-\theta)=rk=v\lambda_2$.
Consequently, the candidate degrees are
\[
    k_\Gamma=\lambda_2-n\quad\text{if }\theta=\theta_2,
    \qquad
    k_\Gamma=\theta_1+\frac{v\lambda_2}{v-n-\theta_1}
       \quad\text{if }\theta=\theta_1.
\]
In the second case, $v-n>\theta_1$, since the smaller eigenvalue of $Q$ is less than either diagonal entry.  
Each candidate must also satisfy \eqref{eqn:degree-compatibility}.
\end{remark}

\begin{lemma}\label{lem:auxSpec}
    Let $\mathcal D$ be a semiregular group divisible design with parameters $(v,b; r,k; \lambda_1, \lambda_2; m, n)$.
    Suppose $\Gamma$ is an auxiliary graph of $\mathcal D$.
    Then 
    \[
    \operatorname{Spec}(A(\Gamma)) = \left \{ [k_\Gamma]^1, [\lambda_2-\lambda_1 - n]^{m(n-1)}, [\lambda_2-\lambda_1]^x, [-n]^y  \right \},
    \]
    where $x = \frac{n(b-1) - k_{\Gamma} - m(\lambda_2 - \lambda_1)(n-1)}{\lambda_2 - \lambda_1 + n}$ and $y = \frac{(\lambda_2 - \lambda_1)(b-1) + k_{\Gamma} - nm(n-1)}{\lambda_2 - \lambda_1 + n}$.
    Moreover,
    \begin{equation}\label{eqn:trace}
    \begin{aligned}
        k_\Gamma^2+m(n-1)(\lambda_2-\lambda_1-n)^2
           +x(\lambda_2-\lambda_1)^2+yn^2=bk_\Gamma.
    \end{aligned}
    \end{equation}
\end{lemma}
\begin{proof}
    By \eqref{eqn:specbtb}, the matrix $B^{\transpose}B$ has the spectrum $[rk]^1,[r-\lambda_1]^{m(n-1)},[0]^{b-m(n-1)-1}$.

Multiplying \eqref{eqn:aux-mixed} on the left by $B^\transpose$ and taking transposes, we obtain $A(\Gamma)B^{\transpose}B = (\lambda_2 - \lambda_1 - n)B^{\transpose}B + \frac{k^2}{v}(k_{\Gamma} + \lambda_1 - \lambda_2 + n) J$.
We consider the eigenvalue $r-\lambda_1$ of $B^{\transpose}B$ which has multiplicity $m(n-1)$.
Any eigenvector in this eigenspace is orthogonal to $\mathbf{1}$. 
Hence, $A(\Gamma)$ has the eigenvalue $\lambda_2 - \lambda_1 - n$ on this eigenspace.

Now we use \eqref{eqn:aux-block} on the nullspace of $B^{\transpose}B$.
This eigenspace (for eigenvalue $0$) is orthogonal to $\mathbf{1}$.
Hence, the eigenvalue of $A(\Gamma)$ on this eigenspace satisfies the quadratic equation $\theta^2 - (\lambda_2 - \lambda_1 - n)\theta - n(\lambda_2 - \lambda_1) = 0$.
Thus, the corresponding eigenvalues of $A(\Gamma)$ are $\lambda_2 - \lambda_1$ and $-n$.

Therefore, $A(\Gamma)$ has the spectrum $[k_\Gamma]^1,[\lambda_2-\lambda_1-n]^{m(n-1)},[\lambda_2-\lambda_1]^x,[-n]^y$ which implies $1 + m(n-1) + x + y = b$.
Furthermore, since $\tr{A(\Gamma)} = 0$, we have $k_{\Gamma} + m(n-1)(\lambda_2 - \lambda_1 - n) + x(\lambda_2 - \lambda_1) + y(-n) = 0$.
One can then solve to find $x = \frac{n(b-1) - k_{\Gamma} - m(\lambda_2 - \lambda_1)(n-1)}{\lambda_2 - \lambda_1 + n}$ and $y = \frac{(\lambda_2 - \lambda_1)(b-1) + k_{\Gamma} - nm(n-1)}{\lambda_2 - \lambda_1 + n}$, as required.
Finally, $\tr(A(\Gamma)^2)=bk_\Gamma$ because $\Gamma$ is $k_\Gamma$-regular.  
Evaluating this trace from the spectrum proves \eqref{eqn:trace}.
\end{proof}

\subsection{Feasible parameters and known constructions}
\label{sub:feasible}

We enumerate semiregular-GDD parameter tuples $(v,b;r,k;\lambda_1,\lambda_2;m,n)$ satisfying necessary conditions for Theorem~\ref{thm:gdd-to-3ev}.  
We consider tuples $(v,n,k,\lambda_2)\in\mathbb N^4$ such that $2\leqslant n,k<v\leqslant100$, $n\mid v$, and $1\leqslant\lambda_2\leqslant50$.  
The relations of Section~\ref{sec:gdd} determine
\[
    m=\frac vn,\qquad r=\frac{v\lambda_2}{k},\qquad
    b=\frac{vr}{k},\qquad \lambda_1=\frac{n\lambda_2-r}{n-1}.
\]
We retain only nonnegative integral parameters with $m\mid k$ and $\lambda_2>\lambda_1$.  
For each remaining tuple, we consider the at most two candidate degrees from the quotient equations in Remark~\ref{rem:degree-compatibility}, retaining only integers $0\leqslant k_\Gamma\leqslant b-1$.
For each candidate we impose the degree condition of Remark~\ref{rem:degree-compatibility} and the trace identity of Lemma~\ref{lem:auxSpec}.  
We require the multiplicities $x,y$ of Lemma~\ref{lem:auxSpec} to be nonnegative integers, and the coefficients $c=(k/v)(k_\Gamma-\theta_1-\theta_2)$ and $c+\theta_1+\theta_2$ in \eqref{eqn:aux-mixed} to be nonnegative integers.
We also require the three full-graph eigenvalues to be integral, with positive integral multiplicities determined by the order and trace, and impose Theorem~\ref{thm:bireg}\,(iii).  
The latter is checked in the form
\[
    \frac{k}{v}(v+k_\Gamma-\theta_1)=-\theta(\theta'+1),
    \qquad \{\theta,\theta'\}=\{\theta_1,\theta_2\},
\]
where $\theta$ is the nonprincipal eigenvalue of the candidate quotient matrix.
Finally, we discard parameter sets with $(v-n)+r=k_\Gamma+k$, for which the resulting graph would be regular rather than biregular.
Exactly one regular parameter set passes the preceding checks; it is discarded: $(9,27;9,3;0,3;3,3)$ with $k_\Gamma=12$, which would give a strongly regular graph with parameters $(36,15,6,6)$.

\begin{table}[htbp]
    \centering
    \caption{Feasible parameter sets for biregular graphs in $\mathscr{G}_3$ generated by semiregular GDDs and complete multipartite subgraphs (Theorem~\ref{thm:gdd-to-3ev}) for $v \leqslant 100$ and $\lambda_2\leqslant50$.}
    \label{tab:feasible-parameters}
    \resizebox{\textwidth}{!}{
    \begin{tabular}{rrrrrrrrrrrrl}
        \toprule
        \multicolumn{8}{c}{\textbf{GDD Parameters}} & \multicolumn{3}{c}{\textbf{Degrees}} & \multicolumn{1}{c}{\textbf{Spectrum}} & \multicolumn{1}{c}{\textbf{Status}} \\
        \cmidrule(lr){1-8} \cmidrule(lr){9-11} \cmidrule(lr){12-12} \cmidrule(lr){13-13}
        $v$ & $b$ & $r$ & $k$ & $\lambda_1$ & $\lambda_2$ & $m$ & $n$ & $k_1$ & $k_2$ & $k_\Gamma$ & $[\theta_0]^1,\, [\theta_1]^{m_1},\, [\theta_2]^{m_2}$ & \\
        \midrule
        $12$ & $54$ & $18$ & $4$ & $0$ & $6$ & $4$ & $3$ & $27$ & $34$ & $30$ & $[33]^1,\, [6]^{18},\, [-3]^{47}$ & Open (Rem~\ref{rem:OA}) \\
        $14$ & $8$ & $4$ & $7$ & $0$ & $2$ & $7$ & $2$ & $16$ & $7$ & $0$ & $[14]^1,\, [2]^7,\, [-2]^{14}$ & Exists (Rem~\ref{rem:VanDamAffine}) \\
        $18$ & $108$ & $72$ & $12$ & $45$ & $48$ & $2$ & $9$ & $81$ & $51$ & $39$ & $[57]^1,\, [3]^{89},\, [-9]^{36}$ & Exists (Rem~\ref{rem:switch126}) \\
        $27$ & $54$ & $18$ & $9$ & $0$ & $6$ & $9$ & $3$ & $42$ & $24$ & $15$ & $[33]^1,\, [6]^{23},\, [-3]^{57}$ & Exists (Thm~\ref{thm:row4-muzychuk}) \\
        $36$ & $48$ & $24$ & $18$ & $8$ & $12$ & $9$ & $4$ & $56$ & $26$ & $8$ & $[44]^1,\, [4]^{36},\, [-4]^{47}$ & Exists (Cor~\ref{cor:main-three}) \\
        $39$ & $27$ & $9$ & $13$ & $0$ & $3$ & $13$ & $3$ & $45$ & $13$ & $0$ & $[39]^1,\, [3]^{26},\, [-3]^{39}$ & Exists (Rem~\ref{rem:VanDamAffine}) \\
        $60$ & $96$ & $48$ & $30$ & $16$ & $24$ & $15$ & $4$ & $104$ & $50$ & $20$ & $[80]^1,\, [8]^{45},\, [-4]^{110}$ & Exists (Thm~\ref{thm:quadrangle-graph}) \\
        $66$ & $120$ & $60$ & $33$ & $24$ & $30$ & $11$ & $6$ & $120$ & $57$ & $24$ & $[90]^1,\, [6]^{85},\, [-6]^{100}$ & Open \\
        $72$ & $135$ & $45$ & $24$ & $9$ & $15$ & $12$ & $6$ & $111$ & $48$ & $24$ & $[84]^1,\, [6]^{96},\, [-6]^{110}$ & Open \\
        $84$ & $64$ & $16$ & $21$ & $0$ & $4$ & $21$ & $4$ & $96$ & $21$ & $0$ & $[84]^1,\, [4]^{63},\, [-4]^{84}$ & Exists (Rem~\ref{rem:VanDamAffine})  \\
        $90$ & $180$ & $90$ & $45$ & $36$ & $45$ & $15$ & $6$ & $174$ & $84$ & $39$ & $[129]^1,\, [9]^{99},\, [-6]^{170}$ & Open \\
        $90$ & $180$ & $90$ & $45$ & $40$ & $45$ & $9$ & $10$ & $170$ & $80$ & $35$ & $[125]^1,\, [5]^{171},\, [-10]^{98}$ & Open \\
        \bottomrule
    \end{tabular}
    }
\end{table}

\begin{remark}
\label{rem:OA}
    The first and fourth rows of Table~\ref{tab:feasible-parameters} both have $\lambda_1 = 0$ and $k = m$, so every block of each is a transversal of the groups by     Remark~\ref{rem:semireg}; that is, their designs are \textbf{transversal designs}, or equivalently \textbf{orthogonal arrays of strength $2$}: arrays over an $n$-set in which every ordered pair of symbols occurs equally often in every pair of columns (see~\cite[Part~III]{handbookCD}).
    The first row requires an $\mathrm{OA}(54,4,3,2)$ and the fourth an $\mathrm{OA}(54,9,3,2)$.
    To obtain these arrays, choose representatives $a_1,\ldots,a_{13}$ of the one-dimensional subspaces of $\FF_3^3$ and form the $27$ rows $(x\cdot a_j)_{j=1}^{13}$, for $x\in\FF_3^3$.
    Two distinct columns are linearly independent, so every ordered symbol pair occurs three times.  
    Duplicating every row and retaining four or nine columns provides the required arrays.
    The required auxiliary graphs are respectively $30$-regular and $15$-regular on $54$ vertices with spectra $\{[30]^1,[6]^9,[3]^8,[-3]^{36}\}$ and $\{[15]^1,[6]^4,[3]^{18},[-3]^{31}\}$.  
    Having four distinct eigenvalues, neither is strongly regular.  
    Section~\ref{sec:row4-muzychuk} constructs a compatible design and auxiliary graph for the fourth row.  
    The first row remains open: an orthogonal array with the required parameters is available, but a compatible auxiliary graph is not known.
    For this row the complement of an auxiliary graph would have Seidel matrix $I+2A(\Gamma)-J$.  
    Its eigenvalue on $\mathbf1$ is $1+2\cdot30-54=7$, and its eigenvalues on $\mathbf1^\perp$ are $1+2\eta$, where $\eta$ runs through the restricted auxiliary eigenvalues.  
    Its spectrum would therefore be $\{[-5]^{36},[7]^9,[13]^9\}$.
    This Seidel spectrum was considered earlier by the first     author~\cite[Table~2]{LAApaper}.
\end{remark}

\begin{remark}
\label{rem:VanDamAffine}
    The three rows of Table~\ref{tab:feasible-parameters} with $k_\Gamma = 0$ form a single family.  
    Writing $q$ for the group size $n$, their parameters are
    \[
        \big(q(q^2{+}q{+}1),\ q^3;\ q^2,\ q^2{+}q{+}1;\ 0,\ q;\ q^2{+}q{+}1,\ q\big)
    \]
    with $q = 2,3,4$ respectively.
      Taking the auxiliary graph to be empty yields the graphs of Van Dam~\cite[\S2.3]{VANDAM1998101}.
\end{remark}

Let $\Gamma = (V,E)$ be a graph and let $\sigma = \{U,V\setminus U\}$ be a partition of the vertex set.
Writing the adjacency matrix of $\Gamma$ in block form as
\[
A(\Gamma) = \begin{bmatrix}
    A(\Gamma[U]) & M \\
    M^\transpose & A(\Gamma[V\setminus U])
\end{bmatrix},
\]
define the graph $\Gamma^\sigma$ by
\[
A(\Gamma^\sigma) = \begin{bmatrix}
    A(\Gamma[U]) & J-M \\
    J-M^\transpose & A(\Gamma[V\setminus U])
\end{bmatrix}.
\]
We say that $\Gamma^\sigma$ is obtained by \textbf{switching} $\Gamma$ with respect to $\sigma$.
Muzychuk and Klin~\cite{MUZYCHUK1998191} used the following tool to study the spectrum of a switched graph.

\begin{proposition}[{\cite[Corollary 3.2]{MUZYCHUK1998191}}]
\label{pro:specSwitch}
    Let $\sigma$ be an equitable $2$-partition of a graph $\Gamma$.
    Then
    \[
    \Spec \left (A(\Gamma^{\sigma}) \right ) =  \Spec(A(\Gamma)) \cup \Spec(Q_\sigma(\Gamma^{\sigma})) - \Spec(Q_\sigma(\Gamma)) .
    \]
\end{proposition}

\begin{remark}
\label{rem:switch126}
    The third row of Table~\ref{tab:feasible-parameters} is obtained by switching a strongly regular graph $\Delta$ with parameters $(126,45,12,18)$, whose spectrum is $\{[45]^1,[3]^{90},[-9]^{35}\}$.
    For completeness, such a pair $(\Delta,U)$ can be specified in coordinates.  
    Take the standard orthogonal graph $\operatorname{NO}_6^-(3)$ ~\cite{BrouwerVanMaldeghem}: its vertices are the projective points $\langle x\rangle$ with $x\cdot x=1$ in $\FF_3^6$, and adjacency is orthogonality for the standard dot product.  
    Let
    \[
        T=\langle e_3+e_4+e_5,\ e_3-e_4+e_6\rangle,
        \qquad
        U_i=\{\langle e_i+t\rangle:t\in T\}\quad(i=1,2).
    \]
    The space $T$ is totally isotropic and perpendicular to $e_1,e_2$.
    Each $U_i$ consists of nine distinct norm-one points.
Since $T$ is totally isotropic and perpendicular to $e_1,e_2$,
\[
    (e_i+s)\cdot(e_j+t)=
    \begin{cases}
        1,& i=j,\\
        0,& i\ne j,
    \end{cases}
    \qquad(s,t\in T;\ i,j\in\{1,2\}).
\]
Thus there are no edges within either $U_i$, and every vertex of $U_1$ is adjacent to every vertex of $U_2$.
Hence $U=U_1\cup U_2$ induces $K_{9,9}=\overline{2K_9}$.
    The partition $\sigma=\{U,V(\Delta)\setminus U\}$ is necessarily equitable.
    Indeed, set $C=A(\Delta)$ and $d_x=|N_\Delta(x)\cap U|$ for $x\notin U$.
    Since $C^2=27I-6C+18J$ and $\Delta[U]$ is $9$-regular, we have
    \[
        \sum_{x\notin U}d_x=18(45-9)=648,\qquad
        \sum_{x\notin U}d_x^2
        =27\cdot18-6\cdot18\cdot9+18\cdot18^2-18\cdot9^2=3888.
    \]
    It follows that $\sum_{x\notin U}(d_x-6)^2=3888-12\cdot648+36\cdot108=0$.
    Thus every vertex outside $U$ has exactly $6$ neighbours in $U$, giving
    \[
        Q_\sigma(\Delta) = \begin{bmatrix} 9 & 36 \\ 6 & 39 \end{bmatrix},
        \qquad
        Q_\sigma(\Delta^\sigma) = \begin{bmatrix} 9 & 72 \\ 12 & 39 \end{bmatrix},
    \]
    These matrices have spectra $\{45,3\}$ and $\{57,-9\}$, respectively.  
    Hence Proposition~\ref{pro:specSwitch} yields
    \[
        \Spec(A(\Delta^\sigma))=\{[57]^1,[3]^{89},[-9]^{36}\}.
    \]
    Thus $\Delta^\sigma\in\mathscr G_3(57,3,-9)$, with degrees $81$ and $51$.
    Switching does not alter $\Delta[U]$, so $\Delta^\sigma[U] = \overline{2K_9}$ is complete multipartite and Theorem~\ref{thm:2-des} applies, returning a group divisible design with $m = 2$, $n = 9$, $r = 72$, $k = 12$, $\lambda_1 = 72 + 3\cdot(-9) = 45$ and $\lambda_2 = 72 + 4\cdot(-8) - 1 + 9 = 48$, which are the parameters listed.
    Its auxiliary graph is $\Delta[V(\Delta)\setminus U]$, of degree $39$ on $108$ vertices, with spectrum $\{[39]^1,[3]^{73},[-6]^{16},[-9]^{18}\}$ by Lemma~\ref{lem:auxSpec}.
\end{remark}

%% ============================================================================
\section{Multiplicative graphs}
\label{sec:multiplicative}

We now specialise the correspondence of Section~\ref{sec:gdds} to multiplicative graphs.  
We first determine the restrictions on the GDD parameters and the spectrum of an auxiliary graph, without choosing a construction.  
The case $n=s^2$, where $s=k/m$, will give the parameters used in Section~\ref{sec:design}.

\subsection{The multiplicative identities}

A graph $\Gamma\in\mathscr G_3(\theta_0,\theta_1,\theta_2)$ is called \textbf{multiplicative} if $\theta_1=-\theta_2$.  
For the complete multipartite construction of Theorem~\ref{thm:gdd-to-3ev}, this is equivalent to $\lambda_2=\lambda_1+n$.  
We call a semiregular GDD with this property \textbf{multiplicative}.
Six of the twelve rows of Table~\ref{tab:feasible-parameters} are of this kind.  
In this case, let $G$ be the auxiliary graph and write $A(G)$ for its adjacency matrix.  
The two matrix identities of Theorem~\ref{thm:gdd-to-3ev} become
\[
    BA(G)=\frac{k k_\Gamma}{v}J,
    \qquad A(G)^2+B^\transpose B=n^2I+\alpha_2^2J,
\]
where $r-\lambda_1=n(\lambda_2-\lambda_1)=n^2$ and $\alpha_2^2=k_\Gamma+k-n^2$.
The scalar compatibility condition~\eqref{eqn:aux-scalar} must also hold.
By \eqref{eqn:specbtb}, $B^\transpose B$ has eigenvalues $n^2$ and $0$ only on $\mathbf1^\perp$.  
Hence $n^{-1}A(G)$ acts as an involution on $\ker B$ and as zero on the centred row space $\operatorname{row}(B)\cap\mathbf1^\perp$.

\subsection{Parameters and the auxiliary spectrum}

Multiplicativity also reduces the number of free parameters.
Indeed, let $s \coloneqq k/m$ denote the number of points in which each block meets each group.
Then all eight GDD-parameters are determined by $m$, $n$, $\lambda_2$ and $s$, namely
\begin{equation}
\label{eqn:BSB}
    v = mn, \qquad b = \frac{n^2\lambda_2}{s^2}, \qquad r = \frac{n\lambda_2}{s}, \qquad
    k = sm, \qquad \lambda_1 = \frac{n(s-1)\lambda_2}{(n-1)s} .
\end{equation}
Imposing multiplicativity, $\lambda_2 = \lambda_1+n$, and eliminating $\lambda_1$ and $\lambda_2$ from \eqref{eqn:BSB} reduces to the single relation $r(n-s) = n^2(n-1)$, whence
\begin{equation}
\label{eqn:params-ns}
    r = \frac{n^2(n-1)}{n-s}, \quad
    \lambda_1 = \frac{n^2(s-1)}{n-s}, \quad
    \lambda_2 = \frac{ns(n-1)}{n-s}, \quad
    b = \frac{n^3(n-1)}{s(n-s)} ,
\end{equation}
together with $v = mn$ and $k = ms$, where $1\leqslant s<n$.

\begin{lemma}
\label{lem:mult-aux}
    Let $\mathcal D$ be a multiplicative semiregular group divisible design and let $\Gamma$ be an auxiliary graph of $\mathcal D$.
    Suppose $\mathcal D$ has parameters $(v,b;r,k;\lambda_1,\lambda_2;m,n)$.
    Then $k_\Gamma=\lambda_1$, or $m>2$ and     $k_\Gamma=\frac{2\lambda_2(v-n)}{v-2n}-\lambda_1$.
    Moreover,
    \begin{equation*}
    m = \frac{k_\Gamma^2-b\,k_\Gamma+n^2(b-1)}{n^2(n-1)}, \qquad x + y = b - 1 - m(n-1), \qquad y - x = \frac{k_\Gamma}{n},
    \end{equation*}
    and, 
    \[
        \Spec(A(\Gamma)) = \left\{ [k_\Gamma]^1,\ [0]^{m(n-1)},\ [n]^{x},\ [-n]^{y} \right\}.
    \]
\end{lemma}

\begin{proof}
By Lemma~\ref{lem:auxSpec} with $\theta_1 = n$ and $\theta_2 = -n$, we have the stated spectrum, since $\theta_1+\theta_2 = 0$ and $\theta_1-\theta_2 = 2n$.
Adding and subtracting the expressions for $x$ and $y$ in Lemma~\ref{lem:auxSpec}, one obtains the stated sum and difference.
Since $x$ and $y$ are integers, so is $y-x = k_\Gamma/n$.

By Remark~\ref{rem:degree-compatibility}, $k_\Gamma=\lambda_2-n=\lambda_1$ or
\[
    k_\Gamma=n+\frac{v\lambda_2}{v-2n}
    =\frac{2\lambda_2(v-n)}{v-2n}-\lambda_1.
\]
In the latter case, Remark~\ref{rem:degree-compatibility} also implies $v-n>n$, so $v=mn$ implies $m>2$.

Multiplicativity implies that $\theta_1+\theta_2 = 0$ and $\theta_1^2 = \theta_2^2 = n^2$, so the trace identity of Lemma~\ref{lem:auxSpec} reads $k_\Gamma^2 + n^2(x+y) = bk_\Gamma$; substituting $x+y = b-1-m(n-1)$ and solving for $m$ yields the stated formula.
\end{proof}

If the graph constructed in Theorem~\ref{thm:gdd-to-3ev} is biregular and its nonprincipal quotient eigenvalue is $n$, then, with $s=k/m$,
\begin{equation}\label{eqn:mult-bireg-branch}
    sm\left(m-2+\frac{\lambda_2}{n}\right)
      =n(n-1)(m-2).
\end{equation}
Indeed, by Theorem~\ref{thm:bireg}\,(iii), we have $\alpha_1\alpha_2=n(n-1)$, while the mixed block of \eqref{eqn:3ev} implies $\alpha_1\alpha_2=(s/n)(mn+k_\Gamma-n)$.
For this quotient branch, Lemma~\ref{lem:mult-aux} implies $m>2$ and $k_\Gamma=n+m\lambda_2/(m-2)$.  
Substituting this value and multiplying by $m-2$ proves \eqref{eqn:mult-bireg-branch}.

Now we show that the case when $s=1$ corresponds to the construction in Remark~\ref{rem:VanDamAffine} when $n$ is a prime power.

\begin{proposition}\label{prop:s-one}
Let $\mathcal D$ be a multiplicative semiregular GDD with parameters $(v,b;r,k;\lambda_1,\lambda_2;m,n)$ and $s=k/m=1$.
Suppose $\Gamma$ be an auxiliary graph of $\mathcal D$ and $v-n+r\ne k_\Gamma+k$.
Then $\Gamma$ is empty and
\[
    (v,b;r,k;\lambda_1,\lambda_2;m,n)
      =\bigl(n(n^2+n+1),n^3;n^2,n^2+n+1;0,n;n^2+n+1,n\bigr).
\]
\end{proposition}

\begin{proof}
Using \eqref{eqn:params-ns}, we find that $\lambda_1=0$, $\lambda_2=n$, $r=n^2$, $k=m$ and $b=n^3$.
Suppose that $k_\Gamma\ne0$.  
By Lemma~\ref{lem:mult-aux}, we have $m>2$ and $k_\Gamma=2n(m-1)/(m-2)$, and
\eqref{eqn:mult-bireg-branch} becomes
\begin{equation}\label{eqn:s1-factor}
    m^2-(n^2-n+1)m+2n(n-1)=0.
\end{equation}
Its discriminant is $(n^2-n-3)^2-8$.  
This is negative for $n=2$; for $n\geqslant4$ it lies strictly between the consecutive squares $(n^2-n-4)^2$ and $(n^2-n-3)^2$.
Thus $n=3$, and \eqref{eqn:s1-factor} implies $m=3$ or $4$.
If $m=3$, then $k_\Gamma=12$ and both full-graph degrees are $15$, contrary to biregularity.  
If $m=4$, then $k_\Gamma=9$; by Lemma~\ref{lem:mult-aux}, we have $x+y=18$ and $y-x=3$, contrary to integrality of the multiplicities.

Hence $k_\Gamma=0$.  
By Lemma~\ref{lem:mult-aux}, we have $m=(n^3-1)/(n-1)=n^2+n+1$.
Substitution provides the remaining parameters.
\end{proof}

The fifth row of Table~\ref{tab:feasible-parameters} has $n=s^2$.
We impose this relation to obtain parameters suited to an affine-plane construction.  
Taking $n=s^2$ in \eqref{eqn:params-ns} yields $\lambda_1=s^3$, $\lambda_2=s^2(s+1)$, $r=s^3(s+1)$, $k=sm$, and $b=s^4(s+1)$.
The next proposition shows that biregularity determines both $m$ and the auxiliary-graph degree, and hence its spectrum.

\begin{proposition}
\label{prop:kG-forced}
    Let $s\geqslant2$ be an integer, and let $\mathcal D$ be a group divisible design with parameters
    \[
(v,b;r,k;\lambda_1,\lambda_2;m,n)  = (s^2m,\ s^4(s+1);\ s^3(s+1),\ sm;\ s^3,\ s^2(s+1);\ m,\ s^2).
\]
    Suppose $\Gamma$ is an auxiliary graph of $\mathcal D$ and $k_\Gamma\ne s(s-1)m+s^2(s^2+s-1)$.
    Then $k_\Gamma = \lambda_1 = s^3$, $m = s^3+1$, and
    \[
        \Spec(A(\Gamma)) = \Big\{ [s^3]^1,\ [0]^{(s^3+1)(s^2-1)},\
        [s^2]^{\frac{s(s+1)(s^2-1)}{2}},\ [-s^2]^{\frac{s(s+1)(s^2-1)}{2}+s} \Big\} .
    \]
\end{proposition}

\begin{proof}
Suppose first that the nonprincipal quotient eigenvalue is $s^2$.
Taking $n=s^2$ and $\lambda_2=s^2(s+1)$ in \eqref{eqn:mult-bireg-branch}, we have $m^2-(s^3-2s+1)m+2s(s^2-1)=0$.
The discriminant is $D_s=s^6-4s^4-6s^3+4s^2+4s+1$.
For $s=2$ it is negative.  
For $s\geqslant3$, we have $(s^3-2s-4)^2<D_s<(s^3-2s-3)^2$, since the differences are $2s^3-12s-15>0$ and $8(s+1)>0$, respectively. 
Thus $D_s$ is not a square, contradicting $m\in\ZZ$.

The nonprincipal quotient eigenvalue is therefore $-s^2$.
By Remark~\ref{rem:degree-compatibility}, $k_\Gamma=\lambda_2-s^2=s^3$.
By Lemma~\ref{lem:mult-aux}, we have $m=s^3+1$ and
\[
    x+y=s^2(s^2+s-1),\qquad y-x=s.
\]
These yield the stated spectrum.
\end{proof}

For the construction, take $s=q$, where $q$ is a prime power.
By Proposition~\ref{prop:kG-forced}, we seek a GDD with parameters
\begin{equation*}
    (v,b;r,k;\lambda_1,\lambda_2;m,n)
    =\big(q^2(q^3+1),\ q^4(q+1);\ q^3(q+1),\ q(q^3+1);\
    q^3,\ q^2(q+1);\ q^3+1,\ q^2\big)
\end{equation*}
and an auxiliary graph with spectrum $\big\{[q^3]^1,[0]^{(q^2-1)(q^3+1)},[q^2]^p,[-q^2]^{p+q}\big\}$, where $p=q(q+1)(q^2-1)/2$.

Muzychuk~\cite{Muzychuk2007} provides candidates for the auxiliary graphs.

%% ============================================================================
\section{An infinite family from Muzychuk's Construction A}
\label{sec:design}

\subsection{Affine planes and parallel classes}
\label{sub:muzychuk-affine-data}

For a prime power $q$, the affine plane $\operatorname{AG}(2,q)$ has point set $\FF_q^2$ and lines the cosets of its one-dimensional subspaces.
The lines with a fixed direction form a \textbf{parallel class}: they partition the point set into $q$ lines of size $q$.
There are $q+1$ parallel classes.  
Lines from distinct classes meet in exactly one point, and two distinct points lie on exactly one line.
These facts follow respectively from the direct-sum decomposition of $\FF_q^2$ by two distinct one-dimensional subspaces and from the direction spanned by the difference of two points.

The parallel classes can be described by linear forms.  
Choose nonzero linear forms $\ell_t:\FF_q^2\to\FF_q$, indexed by $t\in\ZZ_{q+1}$, no two of which are proportional.
For example, enumerate $\FF_q=\{\alpha_0,\ldots,\alpha_{q-1}\}$ and let $\ell_q(x,y)=y$ and, for each $t \in \{0,\dots,q-1\}$, let $\ell_t(x,y)=x+\alpha_t y$.
For each $t$ and $s\in\FF_q$, the fibres $\ell_t^{-1}(s)=\{w\in\FF_q^2:\ell_t(w)=s\}$ are the cosets of the one-dimensional subspace $\ker\ell_t$, and hence form a parallel class of $\operatorname{AG}(2,q)$.
Nonproportional forms have distinct kernels, so these give all $q+1$ parallel classes.

\subsection{Muzychuk's Construction A}
\label{sub:muzychuk-A-overview}\label{sub:aux}

Let $\mathcal P$ be the set of the $q^2+1$ parallel classes of $\operatorname{AG}(2,q^2)$ on $\FF_{q^2}^2$.
For each $a\in\ZZ_{q+1}$, define $\Omega_a:=\{a\}\times\FF_{q^2}^2$.
Thus, each $\Omega_a$ is a labelled copy of the affine plane of order $q^2$; the plane $\operatorname{AG}(2,q)$ introduced above will be used separately to index unions of lines.
For each $a\in\ZZ_{q+1}$, independently choose a set $\mathcal S_a\subseteq\mathcal P$ of $q$ parallel classes and a bijection $\ZZ_{q+1}\setminus\{a\}\longrightarrow\mathcal S_a$ with $b\longmapsto P_{ab}$, where $\mathcal S_a = \{ P_{ab} \; : \; \ZZ_{q+1}\setminus\{a\} \}$.
Write $\mathcal F_a=\mathcal P\setminus\mathcal S_a$.
Following Muzychuk~\cite{Muzychuk2007}, we call parallel classes in $\mathcal F_a$ \textbf{free classes}.
For $a\ne b$, we define
\[
C_{ab}:=\left\{\{a\}\times L:L\in P_{ab}\right\}.
\]
For each unordered pair $\{a,b\}$, choose a bijection $\sigma_{ab}:C_{ab}\longrightarrow C_{ba}$, and let $\sigma_{ba}=\sigma_{ab}^{-1}$.

Define $G_q$ on $\bigcup_a\Omega_a$ as follows.
Each $\Omega_a$ is an independent set.  
For each pair of distinct $a,b$, every point of each $L\in C_{ab}$ is adjacent to every point of $\sigma_{ab}(L)$.
Write $A(G_q)$ for the adjacency matrix of $G_q$.
The graph $G_q$ is an instance of Muzychuk's Construction~A with base graph $K_{q+1}$ and planes $\operatorname{AG}(2,q^2)$ \cite[Section~3.1]{Muzychuk2007}.

For each $a\in\ZZ_{q+1}$ and parallel class $C\in\mathcal P$, define
\[
    U_a(C):=\operatorname{span}_{\RR}
       \left\{
           \mathbf1_{\{a\}\times L}
           -\frac1{q^2}\mathbf1_{\Omega_a}:L\in C
       \right\},
\]
where the indicator vectors are indexed by $V(G_q)$.
Thus $U_a(C)$ consists of the vectors supported on $\Omega_a$ that are constant on each line of the copied parallel class and whose coordinates sum to zero.

\begin{lemma}[{\cite[Proposition~3.1(a)]{Muzychuk2007}}]
\label{cor:muzychuk-affine-planes}
Let $q$ be a prime power and $p=q(q+1)(q^2-1)/2$.
Then
\[
    \Spec(A(G_q))
      =\big\{[q^3]^1,[0]^{(q^2-1)(q^3+1)},
                   [q^2]^p,[-q^2]^{p+q}\big\}.
\]
Moreover,
\[
    \ker A(G_q)
      =\bigoplus_{a\in\ZZ_{q+1}}
         \ \bigoplus_{C\in\mathcal F_a}U_a(C).
\]
\end{lemma}

% \begin{proof}
% Each vertex has $q^2$ neighbours in each of the other $q$ copies, so
% $G_q$ is $q^3$-regular.  In a fixed copy, the centred indicators of
% the $q^2$ disjoint lines of a class have Gram matrix $q^2I-J$.
% Thus $\dim U_a(C)=q^2-1$.  The spaces belonging to distinct classes
% are orthogonal, because two such lines meet in one point.
% Together with the constant space, they exhaust $\RR^{\Omega_a}$:
% \[
%     1+(q^2+1)(q^2-1)=q^4.
% \]

% On the space of vectors constant on each copy, $\mathsf A_1$ acts
% as $q^2(J_{q+1}-I_{q+1})$, giving eigenvalues $q^3$ once and
% $-q^2$ with multiplicity $q$.
% A vector in $U_a(C)$ sums to zero on every line of every other
% parallel class.  Consequently, $\mathsf A_1$ annihilates every
% free-class space.

% For each unordered pair $\{a,b\}$, the space
% $U_a(P_{ab})\oplus U_b(P_{ba})$ is invariant under $\mathsf A_1$.
% Indeed, if a vector takes values $c_L$ on the lines $L\in C_{ab}$,
% where $\sum_L c_L=0$, its image takes the value $q^2c_L$ on
% $\sigma_{ab}(L)$ in copy $b$ and is zero in every other copy.
% The inverse matching gives the reverse action.
% Thus this invariant space has eigenvalues $q^2$ and $-q^2$, each
% with multiplicity $q^2-1$.
% There are $q(q+1)/2$ pairs and
% $(q+1)(q^2-q+1)=q^3+1$ free classes in total.
% These mutually orthogonal spaces exhaust the whole vertex space,
% giving the spectrum and nullspace.  Finally, a regular graph has as
% many connected components as the multiplicity of its degree, so the
% simple eigenvalue $q^3$ proves connectedness.
% \end{proof}

Note that each summand in the orthogonal direct sum for $\ker A(G_q)$ has dimension $q^2-1$.
In particular, the nullspace depends only on the used classes, and not on their assignments or the matching bijections.

\subsection{Incidence rows from the nullspace}
\label{sub:incidence-classes}

Now we build a GDD from the nullspace of $A(G_q)$.
By Lemma~\ref{cor:muzychuk-affine-planes} if $H\subseteq\Omega_a$ is a union of $q$ copied lines of a free class $C\in\mathcal F_a$, then
\begin{equation*}
    \mathbf1_H-\frac1q\mathbf1_{\Omega_a}
       \in U_a(C)\subseteq\ker A(G_q).
\end{equation*}
We choose such unions in each copy to obtain binary incidence rows, with their intersections prescribed by the following criterion.
Set $m=q^3+1$, $v=mq^2$, and $b=q^4(q+1)$.
The vertices of $G_q$ index the blocks of the required GDD.

\begin{lemma}\label{lem:kernel-completion}
Let $B$ be a $v\times b$ $\{0,1\}$-matrix such that
\begin{align}
    \left(B-\frac1qJ\right)A(G_q)&=O,\label{eqn:family-centred-rows}\\
    \left(B-\frac1qJ\right)\left(B-\frac1qJ\right)^\transpose
       &=q^4I_v-q^2(I_m\otimes J_{q^2}).
        \label{eqn:target-centred-gram}
\end{align}
Then $B$ is the incidence matrix of a semiregular GDD $\mathcal D$ with parameters
\[
    (v,b;r,k;\lambda_1,\lambda_2;m,n)
    =\big(q^2(q^3+1),\ q^4(q+1);\ q^3(q+1),\ q(q^3+1);\
    q^3,\ q^2(q+1);\ q^3+1,\ q^2\big).
\]
Moreover, $G_q$ is an auxiliary graph for $\mathcal D$.
% \begin{equation}\label{eqn:forced-incidence}
%     BA(G_q)=q^2J,\qquad B^\transpose B+A(G_q)^2=q^4I_b+(q^3+q)J.
% \end{equation}
\end{lemma}

\begin{proof}
Since $A(G_q)\mathbf1=q^3\mathbf1$, \eqref{eqn:family-centred-rows} implies $\left(B-\frac1qJ\right)\mathbf1=0$, and hence $B\mathbf1=q^3(q+1)\mathbf1$.
For the indicator $\mathbf g$ of any one of the $m$ groups, \eqref{eqn:target-centred-gram} implies
\[
    \left\|\left(B-\frac1qJ\right)^\transpose\mathbf g\right\|^2=0.
\]
Hence,
\[
    (I_m\otimes J_{q^2})B=qJ,
    \qquad B^\transpose\mathbf1=qm\mathbf1.
\]
Furthermore, expanding \eqref{eqn:target-centred-gram} using the row sums of $B$ yields
\[
    BB^\transpose
      =q^4I_v-q^2(I_m\otimes J_{q^2})+(q^3+q^2)J.
\]
These identities give the claimed GDD parameters, which satisfy $r>\lambda_1$ and $rk=v\lambda_2$.
Thus $\mathcal D$ is semiregular.

Let $E_{\ker}$ be the orthogonal projection onto $\ker A(G_q)$.
By \eqref{eqn:target-centred-gram},
\[
    \left(B-\frac1qJ\right)
    \left(B-\frac1qJ\right)^\transpose
      =q^4\left(I_v-\frac1{q^2}(I_m\otimes J_{q^2})\right).
\]
The matrix in parentheses is an orthogonal projection of rank $m(q^2-1)$.  
Consequently,
\[
    \rank\left(B-\frac1qJ\right)=m(q^2-1),
\]
and all nonzero squared singular values of $B-\frac1qJ$ equal $q^4$.
By \eqref{eqn:family-centred-rows} and the symmetry of $A(G_q)$, the row space of $B-\frac1qJ$ is contained in $\ker A(G_q)$.
By Lemma~\ref{cor:muzychuk-affine-planes}, we have $\dim\ker A(G_q)=m(q^2-1)$.
Thus, $\ker A(G_q)$ is equal to the row space of $B-\frac1qJ$.
It follows that
\[
    \left(B-\frac1qJ\right)^\transpose
    \left(B-\frac1qJ\right)=q^4E_{\ker}.
\]
Since $B^\transpose\mathbf1=qm\mathbf1$ and $v=mq^2$, we have $\left(B-\frac1qJ\right)^\transpose\mathbf1=0$.
Hence,
\[
    B^\transpose B
      =\left(B-\frac1qJ\right)^\transpose
         \left(B-\frac1qJ\right)+\frac{v}{q^2}J
      =q^4E_{\ker}+(q^3+1)J.
\]

By Lemma~\ref{cor:muzychuk-affine-planes}, we have
\[
    A(G_q)^2=q^4(I_b-E_{\ker})+(q-1)J.
\]
Together with \eqref{eqn:family-centred-rows} and the $q^3$-regularity of $G_q$, these identities imply
\[
    BA(G_q)=q^2J,
    \qquad
    B^\transpose B+A(G_q)^2=q^4I_b+(q^3+q)J.
\]
Here $\lambda_2-\lambda_1=n=q^2$, so these are precisely \eqref{eqn:aux-mixed} and \eqref{eqn:aux-block}.
Finally, \eqref{eqn:aux-scalar} becomes
\[
    (q^4+q^2)^2=(q^5+q^3)(q^3+q),
\]
which is satisfied.
Thus $G_q$ is an auxiliary graph for $\mathcal D$.
\end{proof}

Now we are ready to construct a GDD from $A(G_q)$.
We note in passing that a similar approach to constructing designs has been used earlier by Shrikhande and Raghavarao~\cite{ShrikhandeRaghavarao1963}, with $\operatorname{AG}(2,q)$ on the line labels as indexing design; see also \cite[Construction~1]{StinsonSwansonTrung2013}.

We seek vectors in $\ker A(G_q)$ whose entries belong to $\{1-\frac1q,-\frac1q\}$.  
Adding $\frac1q\mathbf1$ to each vector yields a binary vector, whose transpose can serve as an incidence row.

For each $a\in\ZZ_{q+1}$, choose $q$ distinct lines $L_{a,1},\ldots,L_{a,q}$ belonging to a single parallel class in $\mathcal F_a$, and put
\[
    \mathcal L_a=\bigcup_{j=1}^q\bigl(\{a\}\times L_{a,j}\bigr)
       \subseteq\Omega_a.
\]
By Lemma~\ref{cor:muzychuk-affine-planes}, the vector
\[
    \mathbf z=\sum_{a\in\ZZ_{q+1}}
       \left(\mathbf1_{\mathcal L_a}-\frac1q\mathbf1_{\Omega_a}\right)
\]
belongs to $\ker A(G_q)$.
Since the sets $\Omega_a$ partition $V(G_q)$, every entry of $\mathbf z$ is either $1-\frac1q$ or $-\frac1q$.
We now construct $m=q^3+1$ groups of $q^2$ such vectors, with the inner products prescribed by \eqref{eqn:target-centred-gram}.

Recall that, for each $t\in\ZZ_{q+1}$, the preimages $\ell_t^{-1}(s)$, with $s\in\FF_q$, form a parallel class of $\operatorname{AG}(2,q)$.
As $t$ varies, these give all $q+1$ parallel classes.
We use these partitions in two ways: first to group the lines of each free class into unions of $q$ lines, and then to select one such union for each $a$ and each row.

For each $a$, enumerate its free classes independently as $\mathcal F_a=\{C_{a,c}:c\in\{1,\ldots,q^2-q+1\}\}$.
For each $a\in\ZZ_{q+1}$ and $c\in\{1,\ldots,q^2-q+1\}$, choose a bijection $\eta_{a,c}:C_{a,c}\longrightarrow\FF_q^2$, and let $C_{a,c}(u)$ denote the unique line of $C_{a,c}$ containing $u\in\FF_{q^2}^2$.
For each $t\in\ZZ_{q+1}$, define
\[
    h_{c,a,t}(u) :=\ell_t\bigl(\eta_{a,c}(C_{a,c}(u))\bigr).
\]
For fixed $a,c,t$ and each $s\in\FF_q$, we have
\[
    \{a\}\times h_{c,a,t}^{-1}(s)
      =\bigsqcup_{\substack{L\in C_{a,c}\\
                           \ell_t(\eta_{a,c}(L))=s}}
         \bigl(\{a\}\times L\bigr).
\]
Since $\eta_{a,c}$ is a bijection and $\ell_t^{-1}(s)$ has $q$ elements, this is a union of exactly $q$ sets of the form $\{a\}\times L$.  
Hence, we can partition $\Omega_a$ as
\[
    \Omega_a
      =\bigsqcup_{s\in\FF_q}
         \bigl(\{a\}\times h_{c,a,t}^{-1}(s)\bigr).
\]

For each fixed $(c,t)$, label the rows by $w\in\FF_q^2$ and use $\ell_a(w)$ to select the set $\{a\}\times h_{c,a,t}^{-1}(\ell_a(w))$ in $\Omega_a$.
This ensures that two distinct rows in the same group select the same subset of $\Omega_a$ for exactly one value of $a$.
Indeed, two distinct points of $\operatorname{AG}(2,q)$ lie on a common line in exactly one parallel class.  
Thus, for distinct $w,w'\in\FF_q^2$,
\begin{equation}\label{eqn:unique-agreement}
    \big|\{a\in\ZZ_{q+1}:\ell_a(w)=\ell_a(w')\}\big|=1.
\end{equation}

Accordingly, for each $(c,t,w)\in\{1,\ldots,q^2-q+1\}\times\ZZ_{q+1}\times\FF_q^2$, define the vector $\mathbf z_{c,t,w}$ by
\[
    \mathbf z_{c,t,w}(a,u)=
    \begin{cases}
        1-\frac1q,&\ell_a(w)=h_{c,a,t}(u),\\[2pt]
        -\frac1q,&\text{otherwise}.
    \end{cases}
\]

Define the $\{0,1\}$-matrix $M_q$ by
\[
    [M_q]_{(c,t,w),(a,u)}
       =\mathbf z_{c,t,w}(a,u)+\frac1q,
\]
with rows having the same pair $(c,t)$ placed consecutively.
The row of $M_q-\frac1qJ$ indexed by $(c,t,w)$ is $\mathbf z_{c,t,w}^\transpose$.
By Lemma~\ref{cor:muzychuk-affine-planes},
\[
    \mathbf z_{c,t,w}\in
       \bigoplus_{a\in\ZZ_{q+1}}U_a(C_{a,c})
       \subseteq\ker A(G_q).
\]
Since $A(G_q)$ is symmetric, it follows that
\begin{equation}\label{eqn:constructed-centred-rows}
    \left(M_q-\frac1qJ\right)A(G_q)=O.
\end{equation}

Regard the rows as points and the columns as blocks.
The point set $\{1,\ldots,q^2-q+1\}\times\ZZ_{q+1}\times\FF_q^2$
is partitioned into the groups
\[
    \mathfrak P_{c,t}=\{(c,t,w):w\in\FF_q^2\},
\]
each of size $q^2$.
There are $(q^2-q+1)(q+1)=m$ groups.
We obtain the block
\begin{equation*}
    \mathfrak B_{a,u}
       =\bigsqcup_{\substack{1\leqslant c\leqslant q^2-q+1\\
                            t\in\ZZ_{q+1}}}
          \bigl(\{c\}\times\{t\}\times
             \ell_a^{-1}(h_{c,a,t}(u))\bigr)
\end{equation*}
from the column of $M_q$ indexed by $(a,u)$.
Since each preimage $\ell_a^{-1}(s)$ has $q$ elements, every block meets each group in exactly $q$ points and therefore has size $qm=q(q^3+1)$.

We now claim that
\begin{equation}\label{eqn:constructed-centred-gram}
    \left(M_q-\frac1qJ\right)
    \left(M_q-\frac1qJ\right)^\transpose
       =q^4I_v-q^2(I_m\otimes J_{q^2}).
\end{equation}
To prove this, we compute the inner products of the rows of $M_q$.
For each $a\in\ZZ_{q+1}$, every row has $q^3$ ones among the coordinates indexed by $\Omega_a$, so its total weight is $(q+1)q^3$.
By \eqref{eqn:unique-agreement}, two distinct rows in the same group select the same subset of $\Omega_a$ for exactly one value of $a$, and disjoint subsets for each of the remaining $q$ values.  
Their inner product is therefore $q^3$.

For $(c,t)\ne(d,s)$, any preimage of a value under $h_{c,a,t}$ meets any preimage of a value under $h_{d,a,s}$ in $q^2$ points.
Indeed, if $c=d$, then $t\ne s$, and the two nonproportional linear forms specify a unique line label under $\eta_{a,c}$.
The intersection is therefore one line of $C_{a,c}$, of size $q^2$.
If $c\ne d$, the classes $C_{a,c}$ and $C_{a,d}$ are distinct.
Each of the $q$ selected lines of $C_{a,c}$ meets each of the $q$ selected lines of $C_{a,d}$ once, with all $q^2$ intersection points distinct.  
Summing over $a$, rows in distinct groups have inner product $(q+1)q^2$.

Since every row has weight $(q+1)q^3$ and there are $b=(q+1)q^4$ columns,
\[
    \left(M_q-\frac1qJ\right)
    \left(M_q-\frac1qJ\right)^\transpose
       =M_qM_q^\transpose-(q+1)q^2J_v.
\]
The preceding inner products therefore give diagonal entries $q^4-q^2$, entries $-q^2$ for distinct rows in the same group, and zero entries between groups.
This proves \eqref{eqn:constructed-centred-gram}.

% \subsection{The resulting graphs}
% \label{sub:family-second-identity}\label{sub:family-graphs}

By Lemma~\ref{lem:kernel-completion}, applied to \eqref{eqn:constructed-centred-rows} and \eqref{eqn:constructed-centred-gram}, the matrix $M_q$ is the incidence matrix of a semiregular GDD $\mathcal D_q$ with the stated parameters, and $G_q$ is an auxiliary graph of $\mathcal D_q$.

Write $\Gamma_q$ for the graph with adjacency matrix $\left[\begin{smallmatrix}X_{\mathcal D_q}&M_q\\
M_q^\transpose&A(G_q)\end{smallmatrix}\right]$.
The results above imply the following corollary.

\begin{corollary}\label{cor:main-three}
Let $q$ be a prime power.  
Then $\Gamma_q$ is a connected biregular graph on $2q^5+q^4+q^2$ vertices with degrees $q^5+q^4+q^3$ and $q^4+q^3+q$.  
Its spectrum is
\[
    \{[q^5+q^3+q^2]^1,[q^2]^{f_q},[-q^2]^{g_q}\},
\]
where $f_q=(q-1)(q+1)^2(2q^2-q+2)/2$ and $g_q=2q^5+q^4+q^2-1-f_q$.
\end{corollary}

%% ============================================================================

%% ============================================================================
\section{An 81-vertex graph from Muzychuk's Construction B}
\label{sec:row4-muzychuk}
\label{sub:row4-muzychuk}

The fourth row of Table~\ref{tab:feasible-parameters} has $(v,b;r,k;\lambda_1,\lambda_2;m,n)=(27,54;18,9;0,6;9,3)$.
By Lemma~\ref{lem:auxSpec}, a corresponding auxiliary graph $G$ would have spectrum
\begin{equation*}
    \Spec(A(G))=\{[15]^1,[6]^4,[3]^{18},[-3]^{31}\}.
\end{equation*}

This spectrum can be realised by the following instance of Muzychuk's Construction~B \cite[Section~4]{Muzychuk2007}.  
Take six copies $\Omega_{\varepsilon,i}$ of $\operatorname{AG}(2,3)$, indexed by $(\varepsilon,i)\in\{0,1\}\times\FF_3$, with coordinates $(x,y)\in\FF_3^2$ in each copy.
We identify the ordered vertex set with $\{0,1\}\times\FF_3^3$, whose elements are $(\varepsilon,i,x,y)$, and use lexicographic order with $0<1<2$ in $\FF_3$.

Within each copy, join points lying on different vertical lines.
For every $i,j,c\in\FF_3$, join the line
\[
    y=(j-i)x+c
\]
in $\Omega_{0,i}$ completely to the line with the same equation in $\Omega_{1,j}$. 
There are no other edges.
Define $G$ to be this graph.

This is Muzychuk's Construction~B with base graph $K_{3,3}$.
Each edge of the base graph uses one nonvertical parallel class, while the unused vertical class determines the internal graph $K_{3,3,3}$.

To apply Theorem~\ref{thm:gdd-to-3ev}, we seek incidence rows $\mathbf f^\transpose$ satisfying $A(G)\mathbf f=3\mathbf f+4\mathbf1$.
Equivalently, $\mathbf w=3\mathbf f-\mathbf1$ is a $\{-1,2\}$-vector in the $3$-eigenspace $E_3$ of $A(G)$.  
Since $E_3\perp\mathbf1$, every such vector has eighteen entries equal to $2$ and thirty-six equal to $-1$, and hence squared norm $108$.  
Moreover,
\[
    \langle\mathbf w,\mathbf w'\rangle
       =9\langle\mathbf f,\mathbf f'\rangle-54.
\]
Thus the required inner products between distinct vectors are $-54$ and $0$.

We enumerate the candidate vectors using the orthogonal projection
\begin{equation*}
    P_3=-\frac{(A(G)-6I)(A(G)+3I)-3J}{18}
       =\frac{18I+3A(G)-A(G)^2+3J}{18}
\end{equation*}
onto $E_3$.  
Take $S_3=\{(0,i,x,y):i,x\in\FF_3,\ y\in\{0,1\}\}$, $C_3=(P_3)_{:,S_3}$, $H_3=(P_3)_{S_3,S_3}$.
With $S_3$ ordered lexicographically, the adjacency rules for $G$ imply
\[
    H_3=\frac16 I_9\otimes
        \begin{bmatrix}2&-1\\-1&2\end{bmatrix},
\]
so $H_3$ is nonsingular and the eighteen columns of $C_3$ form a basis of $E_3$.  
Since $P_3^2=P_3$, we have $C_3^\transpose C_3=H_3$.
For $\mathbf w\in E_3$, set $\mathbf x=\mathbf w|_{S_3}$.  
Then $C_3^\transpose\mathbf w=\mathbf x$, and consequently
\begin{equation*}
    \mathbf w=C_3H_3^{-1}\mathbf x,\qquad
    \|\mathbf w\|^2=\mathbf x^\transpose H_3^{-1}\mathbf x.
\end{equation*}
The complete candidate set is therefore
\begin{equation*}
    \mathcal V_3=
    \left\{C_3H_3^{-1}\mathbf x:
        \mathbf x\in\{-1,2\}^{18},\quad
        \mathbf x^\transpose H_3^{-1}\mathbf x=108\right\}
        \cap\{-1,2\}^{54}.
\end{equation*}
To enumerate this set, first apply the norm condition to the coordinates on $S_3$, then check that every coordinate of the reconstructed vector is $-1$ or $2$.  
In fact, $H_3^{-1}=2I_9\otimes(I_2+J_2)$, and each coordinate pair contributes $12$ to the squared norm unless it is $(2,2)$, when it contributes $48$.  
Thus only $3^9$ assignments pass the norm condition.  
The final coordinate test is still necessary.

Define the \textbf{compatibility graph} $\mathcal H_3$ on $\mathcal V_3$ by joining distinct vectors $\mathbf w,\mathbf z$ when $\mathbf w^\transpose\mathbf z\in\{-54,0\}$, and seek a clique of size $27$.  
We find that $|\mathcal V_3|=324$.
The graph $\mathcal H_3$ is $122$-regular and has $19\,764$ edges: $17\,496$ correspond to inner product $0$, and $2\,268$ to inner product $-54$.  
The following lemma justifies the clique-search target and the grouping of the resulting incidence rows.

\begin{lemma}\label{lem:row4-clique-size}
Every clique in $\mathcal H_3$ has size at most $27$.
For a clique of size $27$, the relation defined by inner product $-54$ is the disjoint union of nine triangles.
If $W$ has these clique vectors as its rows, then $M_4=(W+J)/3$ is the incidence matrix of a semiregular GDD with parameters $(27,54;18,9;0,6;9,3)$, with the triangles as its groups.
\end{lemma}

\begin{proof}
Let $W$ have the vectors of an $h$-clique as its rows, and let $D$ be the adjacency matrix of the graph recording inner product $-54$.
Then
\[
    WW^\transpose=108I_h-54D.
\]
This matrix is positive semidefinite and has rank at most $\dim E_3=18$, so every eigenvalue of $D$ is at most $2$.
In each connected component of $D$, an eigenvalue $2$, if present, is the Perron eigenvalue and is simple.
Such a component has at least three vertices.
The nullity of $108I_h-54D$ is therefore at most $h/3$, whence
\[
    \frac{2h}{3}\leqslant\rank(WW^\transpose)\leqslant18.
\]
Thus $h\leqslant27$.
If $h=27$, equality forces nine connected components of order three, each with spectral radius $2$; each is therefore a triangle.

Within each triangle, the sum of its three vectors has squared norm $3\cdot108+6(-54)=0$, so that sum is zero.
The three corresponding rows of $(W+J)/3$ therefore sum to $\mathbf1_{54}^\transpose$.
Thus every column has sum $9$.  
Every row has sum $18$, and the prescribed inner products give intersection $0$ within each group and $6$ between groups.  
These are the stated GDD parameters.
Finally, $18>0$ and $18\cdot9=27\cdot6$, so the design is semiregular.
\end{proof}

Figure~\ref{fig:row4-matrix} displays the binary matrix $M_4$ obtained from a $27$-clique, with columns in the lexicographic vertex order specified above.  
Its corresponding clique vectors are the rows of
$3M_4-J$.  
The nine groups supplied by Lemma~\ref{lem:row4-clique-size} are indicated on the left.
The identities below verify the GDD parameters and the two auxiliary matrix identities.

\begin{figure}[htbp]
\centering
{\footnotesize\ttfamily
\setlength{\tabcolsep}{0pt}
\begin{tabular}{r@{\hskip 5pt}l}
$\pi_{1}$ & 100100100100100100100100100100100100100100100100100100 \\
 & 010010010010010010010010010010010010010010010010010010 \\
 & 001001001001001001001001001001001001001001001001001001 \\
\noalign{\vskip 2pt}
$\pi_{2}$ & 100001010100001010100001010100001010100001010100001010 \\
 & 010100001010100001010100001010100001010100001010100001 \\
 & 001010100001010100001010100001010100001010100001010100 \\
\noalign{\vskip 2pt}
$\pi_{3}$ & 100010001100010001100010001100010001100010001100010001 \\
 & 010001100010001100010001100010001100010001100010001100 \\
 & 001100010001100010001100010001100010001100010001100010 \\
\noalign{\vskip 2pt}
$\pi_{4}$ & 100100100001001001010010010100100100001001001010010010 \\
 & 010010010100100100001001001010010010100100100001001001 \\
 & 001001001010010010100100100001001001010010010100100100 \\
\noalign{\vskip 2pt}
$\pi_{5}$ & 100001010001010100010100001001010100010100001100001010 \\
 & 010100001100001010001010100100001010001010100010100001 \\
 & 001010100010100001100001010010100001100001010001010100 \\
\noalign{\vskip 2pt}
$\pi_{6}$ & 100010001001100010010001100010001100100010001001100010 \\
 & 010001100100010001001100010001100010010001100100010001 \\
 & 001100010010001100100010001100010001001100010010001100 \\
\noalign{\vskip 2pt}
$\pi_{7}$ & 100100100010010010001001001100100100010010010001001001 \\
 & 010010010001001001100100100010010010001001001100100100 \\
 & 001001001100100100010010010001001001100100100010010010 \\
\noalign{\vskip 2pt}
$\pi_{8}$ & 100001010010100001001010100010100001001010100100001010 \\
 & 010100001001010100100001010001010100100001010010100001 \\
 & 001010100100001010010100001100001010010100001001010100 \\
\noalign{\vskip 2pt}
$\pi_{9}$ & 100010001010001100001100010001100010100010001010001100 \\
 & 010001100001100010100010001100010001010001100001100010 \\
 & 001100010100010001010001100010001100001100010100010001 \\

\end{tabular}}
\caption{The binary incidence matrix $M_4$ obtained from a $27$-clique in $\mathcal H_3$.  
The clique vectors are the rows of $3M_4-J$.
The nine groups of size $3$ are indicated on the left.}
\label{fig:row4-matrix}
\end{figure}

\begin{theorem}
\label{thm:row4-muzychuk}
The matrix $M_4$ of Figure~\ref{fig:row4-matrix} is the incidence matrix of a semiregular GDD with parameters $(27,54;18,9;0,6;9,3)$, and $G$ is an auxiliary graph of this design.
Furthermore, $\left [ \begin{smallmatrix}
    J_{27}-I_9\otimes J_3&M_4\\M_4^\transpose&A(G)
\end{smallmatrix}\right ]$ is the adjacency matrix of a biregular graph $\Gamma_{81}$ with degrees $42$ and $24$ and spectrum $\{[33]^1,[6]^{23},[-3]^{57}\}$.
\end{theorem}

\begin{proof}
By direct computation
\begin{align*}
    M_4M_4^\transpose
       &=18I_{27}-6(I_9\otimes J_3)+6J_{27},
          \\
    M_4\mathbf1&=18\mathbf1,\qquad
          M_4^\transpose\mathbf1=9\mathbf1,\\
    M_4A(G)&=3M_4+4J_{27,54},\\
    M_4^\transpose M_4+A(G)^2-3A(G)
       &=18I_{54}+6J_{54}.
\end{align*}
The first two lines give the stated GDD parameters, with the indicated groups.  
The design is semiregular since $18>0$ and $18\cdot9=27\cdot6$.
The last two lines are the auxiliary-graph identities; the coefficient of $J_{54}$ in \eqref{eqn:aux-block} is $((15-6)(15+3)+18\cdot9)/54=6$.
Theorem~\ref{thm:gdd-to-3ev} now applies.  
The quotient matrix $\left[\begin{smallmatrix}24&18\\9&15\end{smallmatrix}\right]$ has eigenvalues $33$ and $6$, and its row sums give the two degrees.
The eigenvalue multiplicities follow from the order of $\Gamma_{81}$ and the trace of its adjacency matrix.
\end{proof}

%% ============================================================================
\section{A design and a graph from a generalised quadrangle}
\label{sec:quadrangle}

The seventh row of Table~\ref{tab:feasible-parameters} requires a semiregular GDD with parameters $(60,96;48,30;16,24;15,4)$ and, by Lemma~\ref{lem:auxSpec}, a strongly regular auxiliary graph with parameters $(96,20,4,4)$.  
We use the point graph $G$ of the dual of $T_2^*(\mathcal O)$, where $\mathcal O$ is a hyperoval of $\operatorname{PG}(2,4)$.
This quadrangle has order $(5,3)$, and
\begin{equation}\label{eqn:hyp-spectrum}
    \Spec(A(G))=\{[20]^1,[4]^{45},[-4]^{50}\},
    \qquad A(G)^2=16I+4J;
\end{equation}
see \cite[Section~1.3]{PayneThas1984}.

For reproducibility, we specify the order of the vertices of $G$ used for the columns of Figure~\ref{fig:hyp-matrix}.  
Write $\FF_4=\{0,1,\omega,\omega^2\}$, where $\omega^2=\omega+1$, and fix the symbol order $0<1<\omega<\omega^2$.  
Choose the hyperoval with direction representatives
\[
\begin{gathered}
 d_1=(1,0,0),\quad d_2=(1,1,1),\quad d_3=(1,\omega,\omega^2),\\
 d_4=(1,\omega^2,\omega),\quad d_5=(0,0,1),\quad d_6=(0,1,0).
\end{gathered}
\]
The vertices are the affine lines in $\FF_4^3$ with one of these six directions; two distinct vertices are adjacent precisely when the corresponding lines intersect.  
The columns are ordered first by $i=1,\ldots,6$ and then lexicographically by $(a,b)\in\FF_4^2$, using the lines
\[
 L_{i,a,b}=
 \begin{cases}
 (0,a,b)+\FF_4d_i,&1\le i\le4,\\
 (a,b,0)+\FF_4d_5,&i=5,\\
 (a,0,b)+\FF_4d_6,&i=6.
 \end{cases}
\]
This specifies all $96$ columns and defines $A(G)$ independently of $B$.

We construct the incidence matrix from $A(G)$, one binary row at a time.

The mixed auxiliary identity requires each incidence row $\mathbf f^\transpose$ to satisfy
\begin{equation*}
    (A(G)-4I)\mathbf f=8\mathbf1.
\end{equation*}
Thus its support is an $8$-tight set in the terminology of \cite{BambergLawPenttila2009}.  
Summing this equation gives $16\mathbf1^\transpose\mathbf f=8\cdot96$, so every binary solution has weight $48$.  
The remaining row constraints are just the intersections required by the design parameters.

Suppose that $\mathcal C$ contains the rows in completed groups and $\mathcal D$ the rows already chosen in the current group.
We seek
\begin{equation}\label{eqn:hyp-direct-constraints}
\begin{aligned}
    \mathbf f&\in\{0,1\}^{96},&
    (A(G)-4I)\mathbf f&=8\mathbf1,\qquad
    \mathbf1^\transpose\mathbf f=48,\\
    \mathbf c^\transpose\mathbf f&=24
       &&\text{for all $\mathbf c\in\mathcal C$},\\
    \mathbf d^\transpose\mathbf f&=16
       &&\text{for all $\mathbf d\in\mathcal D$}.
\end{aligned}
\end{equation}
These are binary linear constraints, since previously chosen rows are fixed.  
We start with $\mathcal C=\mathcal D=\emptyset$, retain a solution in $\mathcal D$, and repeat.  
Whenever $|\mathcal D|=4$, we move its rows into $\mathcal C$ and start the next group.  
The search stops when fifteen groups have been completed.
The redundant weight equation is retained in the computation.

For each row, we choose the lexicographically greatest feasible solution in the chosen coordinate order, with $1$ preferred to $0$.
We maximise the next coordinate over the feasible binary solutions, fix it at its optimum, and continue.  
If a feasible solution already has value $1$ at that coordinate, it supplies the optimum immediately.
The implementation uses \texttt{MixedIntegerLinearProgram} in \textsc{SageMath}, with the GLPK backend \cite{SageMathMILP}.
Every accepted row is checked against \eqref{eqn:hyp-direct-constraints} using integer arithmetic.

The first attempt uses the input coordinate order.  
If a row search fails, we discard the partial design and restart with a random permutation of the coordinates, kept fixed throughout that attempt.
The implementation allows $20$ attempts.  
A solver error aborts the attempt; it is not used to fix a coordinate to $0$.  
Failure within the attempt limit does not establish nonexistence.

Let $B$ have the sixty selected rows, ordered by groups.  
We check all the identities below exactly.  
Figure~\ref{fig:hyp-matrix} displays an incidence matrix with the required properties.  
The column order in the figure is the affine-line order specified
above, which defines $A(G)$ independently of $B$.

\begin{figure}[htbp]
\centering
{\scriptsize\ttfamily
\setlength{\tabcolsep}{0pt}
\begin{tabular}{r@{\hskip 4pt}l}
$\pi_{1}$ & 110011000011001111001100001100111100110000110011110011000011001111111111000000001111111100000000 \\
 & 110011000011001111001100001100110011001111001100001100111100110000000000111111110000000011111111 \\
 & 001100111100110000110011110011001100110000110011110011000011001100000000111111110000000011111111 \\
 & 001100111100110000110011110011000011001111001100001100111100110011111111000000001111111100000000 \\
\noalign{\vskip 2pt}
$\pi_{2}$ & 100101100110100110010110011010011111000011110000101010101010101010101010010101011010101001010101 \\
 & 100101100110100110010110011010010000111100001111010101010101010101010101101010100101010110101010 \\
 & 011010011001011001101001100101101111000011110000010101010101010110101010010101010101010110101010 \\
 & 011010011001011001101001100101100000111100001111101010101010101001010101101010101010101001010101 \\
\noalign{\vskip 2pt}
$\pi_{3}$ & 101001011010010110100101101001011001100110011001111100000000111110011001011001101001100101100110 \\
 & 101001011010010110100101101001010110011001100110000011111111000001100110100110010110011010011001 \\
 & 010110100101101001011010010110101001100110011001000011111111000001100110100110011001100101100110 \\
 & 010110100101101001011010010110100110011001100110111100000000111110011001011001100110011010011001 \\
\noalign{\vskip 2pt}
$\pi_{4}$ & 101010100101010110010110100101101010101001010101100101101001011011110000000011111111000000001111 \\
 & 101010100101010101101001011010011010101001010101011010010110100100001111111100000000111111110000 \\
 & 010101011010101010010110100101100101010110101010100101101001011000001111111100000000111111110000 \\
 & 010101011010101001101001011010010101010110101010011010010110100111110000000011111111000000001111 \\
\noalign{\vskip 2pt}
$\pi_{5}$ & 110000110011110011111111000000001100001100111100110011001100110011001100001100111100001100111100 \\
 & 110000110011110000000000111111111100001100111100001100110011001100110011110011000011110011000011 \\
 & 001111001100001111111111000000000011110011000011001100110011001111001100001100110011110011000011 \\
 & 001111001100001100000000111111110011110011000011110011001100110000110011110011001100001100111100 \\
\noalign{\vskip 2pt}
$\pi_{6}$ & 100101101001011011001100110011001001011010010110111100001111000010100101010110101100110000110011 \\
 & 100101101001011000110011001100111001011010010110000011110000111101011010101001010011001111001100 \\
 & 011010010110100111001100110011000110100101101001000011110000111101011010101001011100110000110011 \\
 & 011010010110100100110011001100110110100101101001111100001111000010100101010110100011001111001100 \\
\noalign{\vskip 2pt}
$\pi_{7}$ & 100110010110011010100101010110101010010101011010100110010110011011110000111100001111000011110000 \\
 & 100110010110011001011010101001010101101010100101100110010110011000001111000011110000111100001111 \\
 & 011001101001100110100101010110101010010101011010011001101001100100001111000011110000111100001111 \\
 & 011001101001100101011010101001010101101010100101011001101001100111110000111100001111000011110000 \\
\noalign{\vskip 2pt}
$\pi_{8}$ & 101001010101101011110000000011111010101010101010101001010101101010010110011010011010010101011010 \\
 & 101001010101101000001111111100000101010101010101101001010101101001101001100101100101101010100101 \\
 & 010110101010010111110000000011110101010101010101010110101010010110010110011010010101101010100101 \\
 & 010110101010010100001111111100001010101010101010010110101010010101101001100101101010010101011010 \\
\noalign{\vskip 2pt}
$\pi_{9}$ & 110000111100001110011001100110011111111100000000110000111100001111000011001111001001011001101001 \\
 & 110000111100001101100110011001100000000011111111110000111100001100111100110000110110100110010110 \\
 & 001111000011110010011001100110010000000011111111001111000011110000111100110000111001011001101001 \\
 & 001111000011110001100110011001101111111100000000001111000011110011000011001111000110100110010110 \\
\noalign{\vskip 2pt}
$\pi_{10}$ & 111111110000000011000011110000111100001111000011100110011001100111001100110011001001011010010110 \\
 & 111111110000000000111100001111000011110000111100011001100110011011001100110011000110100101101001 \\
 & 000000001111111111000011110000111100001111000011011001100110011000110011001100110110100101101001 \\
 & 000000001111111100111100001111000011110000111100100110011001100100110011001100111001011010010110 \\
\noalign{\vskip 2pt}
$\pi_{11}$ & 111100000000111110101010010101011100110011001100101010100101010110011001100110011100001111000011 \\
 & 111100000000111101010101101010100011001100110011010101011010101010011001100110010011110000111100 \\
 & 000011111111000010101010010101010011001100110011101010100101010101100110011001100011110000111100 \\
 & 000011111111000001010101101010101100110011001100010101011010101001100110011001101100001111000011 \\
\noalign{\vskip 2pt}
$\pi_{12}$ & 111100001111000010101010101010101001011001101001100101100110100110101010101010101010010110100101 \\
 & 111100001111000001010101010101010110100110010110011010011001011010101010101010100101101001011010 \\
 & 000011110000111110101010101010100110100110010110011010011001011001010101010101011010010110100101 \\
 & 000011110000111101010101010101011001011001101001100101100110100101010101010101010101101001011010 \\
\noalign{\vskip 2pt}
$\pi_{13}$ & 110011001100110010011001011001101001100101100110111111110000000011000011110000111100110011001100 \\
 & 110011001100110001100110100110010110011010011001000000001111111100111100001111001100110011001100 \\
 & 001100110011001110011001011001101001100101100110000000001111111100111100001111000011001100110011 \\
 & 001100110011001101100110100110010110011010011001111111110000000011000011110000110011001100110011 \\
\noalign{\vskip 2pt}
$\pi_{14}$ & 100110011001100111000011001111001111000000001111110000110011110010010110100101101001100110011001 \\
 & 100110011001100100111100110000110000111111110000001111001100001101101001011010011001100110011001 \\
 & 011001100110011011000011001111000000111111110000110000110011110001101001011010010110011001100110 \\
 & 011001100110011000111100110000111111000000001111001111001100001110010110100101100110011001100110 \\
\noalign{\vskip 2pt}
$\pi_{15}$ & 101010101010101011110000111100001010010110100101101001011010010110100101101001011010101010101010 \\
 & 101010101010101000001111000011110101101001011010010110100101101001011010010110101010101010101010 \\
 & 010101010101010111110000111100000101101001011010010110100101101010100101101001010101010101010101 \\
 & 010101010101010100001111000011111010010110100101101001011010010101011010010110100101010101010101 \\
\end{tabular}}
\caption{An incidence matrix $B$ of a semiregular GDD with parameters $(60,96;48,30;16,24;15,4)$.  
The fifteen groups of four rows are indicated on the left.}
\label{fig:hyp-matrix}
\end{figure}

\begin{theorem}\label{thm:quadrangle-graph}
The matrix $B$ in Figure~\ref{fig:hyp-matrix} is the incidence matrix of a semiregular GDD with parameters $(60,96;48,30;16,24;15,4)$ having $G$ as an auxiliary graph. 
Furthermore, $\left [ \begin{smallmatrix}J_{60}-I_{15}\otimes J_4&B\\B^\transpose&A(G)\end{smallmatrix}\right ]$ is the adjacency matrix of a connected biregular graph with degrees $104$ and $50$ and spectrum
\[
    \{[80]^1,[8]^{45},[-4]^{110}\}.
\]
\end{theorem}

\begin{proof}
Set $W=2B-J$ and $H=I_{15}\otimes J_4$.  
The row constraints give
\[
    WA(G)=4W,\qquad W\mathbf1=0,\qquad
    WW^\transpose=128I_{60}-32H.
\]
For the indicator $\mathbf g$ of any group of four rows, the last identity gives $\|W^\transpose\mathbf g\|^2=0$.  
Thus each group of sign rows sums to zero, and every column of $B$ contains two ones in each group.  
Consequently,
\begin{equation*}
\begin{gathered}
    BB^\transpose=32I_{60}-8H+24J_{60},\\
    B\mathbf1=48\mathbf1,\qquad B^\transpose\mathbf1=30\mathbf1,
    \qquad BA(G)=4B+8J.
\end{gathered}
\end{equation*}
These give the GDD parameters and semiregularity.  
The rank of $WW^\transpose$ is $45$, so the rows of $W$ span
$E_4=\ker(A(G)-4I)$.  
The nonzero eigenvalues of $WW^\transpose$ all equal $128$, so $W^\transpose W=128P_4$, where $P_4$ is the orthogonal projection onto $E_4$.  
Since $W^\transpose\mathbf1=0$ and $B=(W+J)/2$, we have
\begin{equation}\label{eqn:hyp-column-gram}
    B^\transpose B=32P_4+15J=16I_{96}+4A(G)+14J_{96},
    \qquad P_4=\frac{16I+4A(G)-J}{32}.
\end{equation}
Together with \eqref{eqn:hyp-spectrum}, this gives $A(G)^2-4A(G)+B^\transpose B=32I+18J$.
The coefficient in \eqref{eqn:aux-block} is $((20+16-24)(20+4)+48\cdot30)/96=18$, so Theorem~\ref{thm:gdd-to-3ev} applies.
The quotient matrix $\left[\begin{smallmatrix}56&48\\30&20\end{smallmatrix}\right]$ has eigenvalues $80,-4$ and row sums $104,50$; the multiplicities follow from the order and trace of the full adjacency matrix.
\end{proof}

%% ============================================================================

\section{Association schemes and coherent closure}
\label{sec:cr}

In this section, we consider the coherent closure of the graphs $G_q$.
We first recall the required notions of coherent configuration, association scheme, and coherent closure; see also \cite[Section~3]{GreavesYip2024}.
We then determine the four-class schemes on the block set of the infinite family and the type of the coherent closure of the graphs $\Gamma_q$.

\subsection{Coherent closure and association schemes}

A \textbf{coherent algebra} on a finite set $\mathfrak X$ is a real matrix algebra containing $I,J$ and closed under transpose and entrywise multiplication.  
It has a unique basis of nonzero, mutually disjoint binary matrices summing to $J$; their supports are the relations of the corresponding \textbf{coherent configuration}.
These standard facts and the terminology below are recalled in \cite[Section~3]{GreavesYip2024}.

The diagonal basis relations partition $\mathfrak X$ into \textbf{fibres}.  
Each basis relation is supported on a unique ordered pair of fibres.  
The \textbf{type} is the symmetric matrix $(t_{ij})$, where $t_{ij}$ counts the basis relations from fibre $i$ to fibre $j$.
Its entries sum to the dimension, or \textbf{rank}, of the algebra.
A \textbf{strongly regular design}, in the sense of Higman~\cite{Higman1988}, is a coherent configuration with two fibres, two basis relations in each direction between the fibres, and a symmetric rank-three association scheme on each fibre.
It therefore has type $\left[\begin{smallmatrix}3&2\\2&3\end{smallmatrix}\right]$ and rank $10$.
The \textbf{coherent closure} $\mathcal W(\Gamma)$ is the intersection of all coherent algebras containing $A(\Gamma)$; its dimension is the \textbf{coherent rank} $\crk(\Gamma)$.

A configuration with one fibre and symmetric basis relations is a \textbf{symmetric association scheme}.  
If its basis matrices are $\mathsf A_0=I,\mathsf A_1,\ldots,\mathsf A_d$, we call it a $d$-class scheme and its span the \textbf{Bose--Mesner algebra}.
These matrices commute: their products belong to an algebra of symmetric matrices, so $\mathsf A_i\mathsf A_j=(\mathsf A_i\mathsf A_j)^\transpose =\mathsf A_j\mathsf A_i$.
The orthogonal projections onto the common eigenspaces are its \textbf{primitive idempotents} $\mathsf E_0,\ldots,\mathsf E_d$, indexed with $\mathsf E_0=|\mathfrak X|^{-1}J$.
The \textbf{first eigenmatrix} $P$ is defined by
\[
    \mathsf A_j=\sum_{i=0}^d P_{ij}\mathsf E_i;
\]
thus $P_{ij}$ is the eigenvalue of $\mathsf A_j$ on $\operatorname{Im}\mathsf E_i$, whose dimension is $\rank\mathsf E_i$.

\subsection{A four-class association scheme}
\label{sec:fourclass}

Let $\mathsf A_1 := A(G_q)$ be the adjacency matrix of the graph $G_q$, defined in Section~\ref{sub:muzychuk-A-overview}.  
We retain the independently chosen used-class sets $\mathcal S_a$ and their complements $\mathcal F_a$ from that construction.
We now define the remaining relations on its $q^4(q+1)$ vertices and determine the association scheme they form.

For each $a\in\ZZ_{q+1}$, let $L_{q,a}$ be the graph on $\Omega_a$ in which $(a,u)$ and $(a,v)$, with $u\ne v$, are adjacent precisely when the line through $u$ and $v$ in $\operatorname{AG}(2,q^2)$ belongs to a class in $\mathcal S_a$.  Define
\[
\begin{gathered}
    \mathsf A_3=\operatorname{diag}_{a\in\ZZ_{q+1}}A(L_{q,a}),\qquad
    \mathsf A_4=\operatorname{diag}_{a\in\ZZ_{q+1}}A(\overline{L_{q,a}}),\\
    \mathsf A_{\neq}=(J_{q+1}-I_{q+1})\otimes J_{q^4},
\end{gathered}
\]
and set $\mathsf A_2=\mathsf A_{\neq}-\mathsf A_1$.  
Thus $\mathsf A_3$ and $\mathsf A_4$ distinguish whether the line joining two points of the same copy belongs to a used or a free parallel class.  
The five matrices $I,\mathsf A_1,\mathsf A_2,\mathsf A_3,\mathsf A_4$ are symmetric $\{0,1\}$-matrices with pairwise disjoint supports and sum $J$.

The selected parallel classes form a net of degree $q$ and order $q^2$, whose point graph $L_{q,a}$ is a Latin square graph $\operatorname{LS}_q(q^2)$.  
The following spectrum is standard.

\begin{lemma}[{\cite[Section~8.4.2]{BrouwerVanMaldeghem}}]
\label{lem:Cayley}
For each $a\in\ZZ_{q+1}$, the graph $L_{q,a}$ has spectrum
\[
\Spec(A(L_{q,a})) = \left \{ [q^3-q]^1,\ [q^2-q]^{q(q^2-1)},\ [-q]^{(q-1)(q^3+1)} \right \}.
\]
\end{lemma}

Define $\mathcal A_q^\prime\coloneqq
\{I,\mathsf A_{\neq},\mathsf A_3,\mathsf A_4\}$.
We first show that $\mathcal A_q^\prime$ is a $3$-class association scheme.

\begin{lemma}
\label{lem:eigenmatrix-prime}
Let $q$ be a prime power. 
Then $\mathcal A_q^\prime$ is a symmetric three-class association scheme with first eigenmatrix
\[
\bbordermatrix{
    & I & \mathsf A_{\neq} & \mathsf A_3 & \mathsf A_4 \cr
  \mathsf E_0 & 1 & q^5 & q^3-q & q^4-q^3+q-1 \cr
  \mathsf E_1 & 1 & -q^4 & q^3-q & q^4-q^3+q-1 \cr
  \mathsf E_2 & 1 & 0 & -q & q-1 \cr
  \mathsf E_3^\prime & 1 & 0 & q^2-q & -q^2+q-1 \cr
}.
\]
The ranks of $\mathsf E_0, \mathsf E_1, \mathsf E_2, \mathsf E_3^\prime$ are $1$, $q$, $(q^2-1)(q^3+1)$ and $q(q+1)(q^2-1)$ respectively. 
Moreover, the column space of $\mathsf E_1$ is spanned by the vectors $\mathbf{g}\otimes\mathbf{1}_{q^4}$ with $\mathbf{g}\in\RR^{q+1}$ and $\sum_a g_a=0$.
\end{lemma}

\begin{proof}
By Lemma~\ref{lem:Cayley}, the graphs $L_{q,a}$ are strongly regular with parameters independent of $a$.  
Hence products of the within-copy relations have the same coefficients in every copy.  
Products involving $\mathsf A_{\neq}$ depend only on the common order and valency of these graphs.  
Thus the span of $\mathcal A_q^\prime$ is closed under multiplication, so $\mathcal A_q^\prime$ is a symmetric three-class association scheme.

Its common eigenspaces are the span of $\mathbf1$; the space $\{\mathbf g\otimes\mathbf1_{q^4}:\sum_a g_a=0\}$ of dimension $q$; and the two spaces
\[
    \mathfrak U_1 = \bigoplus_{a\in\ZZ_{q+1}}\ \bigoplus_{C\in\mathcal F_a}U_a(C),
    \qquad
    \mathfrak U_2 =\bigoplus_{a\in\ZZ_{q+1}}\ \bigoplus_{C\in\mathcal S_a}U_a(C).
\]
On the copy-constant space orthogonal to $\mathbf1$, the matrices $\mathsf A_{\neq}$ and $\mathsf A_3$ act as $-q^4$ and $q^3-q$, respectively. 
On both the spaces $\mathfrak U_1$ and $\mathfrak U_2$, the matrix $\mathsf A_{\neq}$ acts as zero.
For $C\in\mathcal P$, let $R_{a,C}$ join distinct points of $\Omega_a$ on a common copied line $\{a\}\times L$, where $L\in C$.  
Then $A(L_{q,a})=\sum_{C\in\mathcal S_a}R_{a,C}$.
On $U_a(C')$, the matrix $R_{a,C}$ acts as $q^2-1$ when $C=C'$ and as $-1$ otherwise: lines of distinct classes meet in one point, so a vector in $U_a(C')$ sums to zero on every line of $C$.
Summing over the $q$ used classes, $\mathsf A_3$ therefore acts as $-q$ on the free-class space and as $q^2-q$ on the used-class space.
Their dimensions are $(q+1)(q^2-q+1)(q^2-1)=(q^2-1)(q^3+1)$ and $(q+1)q(q^2-1)$, respectively.
The entries for $\mathsf A_4$ follow from $I + \mathsf A_{\neq}+\mathsf A_3+\mathsf A_4 = J$.
\end{proof}

Let $\mathcal{A}_q=\{I,\mathsf A_1, \mathsf A_2, \mathsf A_3,\mathsf A_4\}$.
We now show that $\mathcal{A}_q$ is a $4$-class association scheme.

\begin{theorem}
\label{thm:main-scheme-and-eigenmatrix}
Let $q$ be a prime power.
Then $\mathcal A_q$ is a four-class symmetric association scheme whose first eigenmatrix is
\[
\bbordermatrix{
    & I & \mathsf A_1 & \mathsf A_2 & \mathsf A_3 & \mathsf A_4 \cr
  \mathsf E_0 & 1 & q^3 & q^5-q^3 & q^3-q & q^4-q^3+q-1 \cr
  \mathsf E_1 & 1 & -q^2 & q^2-q^4 & q^3-q & q^4-q^3+q-1 \cr
  \mathsf E_2 & 1 & 0 & 0 & -q & q-1 \cr
  \mathsf E_3 & 1 & q^2 & -q^2 & q^2-q & -q^2+q-1 \cr
  \mathsf E_4 & 1 & -q^2 & q^2 & q^2-q & -q^2+q-1 
}.
\]
The ranks of $\mathsf E_0, \mathsf E_1, \mathsf E_2, \mathsf E_3, \mathsf E_4$ are respectively $1$, $q$, $(q^2-1)(q^3+1)$, $\frac{q(q+1)(q^2-1)}{2}$ and $\frac{q(q+1)(q^2-1)}{2}$.
\end{theorem}

\begin{proof}
We first verify the products involving $\mathsf A_1$.  
These are instances of the matched parallel-class calculations in \cite[Proposition~2.2]{Muzychuk2007}:
\begin{align}
    \mathsf A_1^2&=q^3I+(q-1)\mathsf A_{\neq}+q^2\mathsf A_3,
        \label{eqn:A1sq}\\
    \mathsf A_1\mathsf A_3&=(q^2-1)\mathsf A_1+(q-1)\mathsf A_2.
        \label{eqn:A1A3}
\end{align}
Every vertex has a line of $q^2$ neighbours in each of the $q$ other copies, so its degree is $q^3$.  
Two distinct vertices in the same copy have common neighbours in another copy precisely when they lie on the same line of the
class assigned to that copy.  
Each selected class is used once, so their number of common neighbours is $q^2$ when their joining line is in a
selected class and zero otherwise.  
For vertices in distinct copies, each of the $q-1$ remaining copies contributes the intersection of two lines from distinct parallel classes, hence one common neighbour.  
This proves \eqref{eqn:A1sq}.

For \eqref{eqn:A1A3}, take vertices $x\in\Omega_a$ and $y\in\Omega_b$, where $a\ne b$.  
The neighbours of $x$ in $\Omega_b$ form a line $L$ in a selected class.  
If $y\in L$, then the $q^2-1$ other points of $L$ are adjacent to $y$ in $L_{q,b}$, and there are no further contributions.  
If $y\notin L$, each of the other $q-1$ selected classes supplies exactly one point of $L$ adjacent to $y$.  
These points are distinct.  
The product is zero when the two vertices lie in the same copy, proving the identity.

Since $\mathsf A_2 = \mathsf A_{\neq} - \mathsf A_1$ and $\mathsf A_4 = J - I - \mathsf A_{\neq} - \mathsf A_3$, and since $\{I, \mathsf A_{\neq}, \mathsf A_3, \mathsf A_4\}$ spans the Bose--Mesner algebra of $\mathcal{A}_q^\prime$, \eqref{eqn:A1sq} and \eqref{eqn:A1A3}, together with
\begin{equation}
    \label{eqn:A1Ane}
        \mathsf A_1\mathsf A_{\neq}
    =q^3(I+\mathsf A_3+\mathsf A_4)+q^2(q-1)\mathsf A_{\neq},
\end{equation}
show that the span of $\mathcal A_q$ is closed under multiplication.
Here, \eqref{eqn:A1Ane} follows from the fact that every vertex has $q^2$ neighbours in each other copy of $\FF_{q^2}^2$.
Thus $\mathcal{A}_q$ is a four-class symmetric association scheme.

Now we determine its first eigenmatrix. By \eqref{eqn:A1sq}, each eigenvalue $\theta$ of $\mathsf A_1$ satisfies $\theta^2 = q^3 + (q-1)\alpha + q^2\beta$, where $\alpha$ and $\beta$ are the eigenvalues of $\mathsf A_{\neq}$ and $\mathsf A_3$ on the corresponding idempotent of $\mathcal A_q'$.
Using the values of Lemma~\ref{lem:eigenmatrix-prime}, we find
\[
    \mathsf A_1\mathsf E_0=q^3\mathsf E_0,\qquad
    \mathsf A_1^2\mathsf E_1=q^4\mathsf E_1,\qquad
    \mathsf A_1\mathsf E_2=O,\qquad
    \mathsf A_1^2\mathsf E_3^\prime=q^4\mathsf E_3^\prime.
\]
Thus the eigenvalues of $\mathsf A_1$ on the column spaces of $\mathsf E_1$ and $\mathsf E_3^\prime$ lie in $\{q^2,-q^2\}$.

To determine the sign on $\mathsf E_1$, let $\mathbf x = \mathbf g \otimes \mathbf{1}_{q^4}$ with $\sum_{a=0}^q g_a = 0$, as in Lemma~\ref{lem:eigenmatrix-prime}.
By construction, for fixed distinct $a,b$ and fixed $\mathbf v$ there are exactly $q^2$ vectors $\mathbf w$ with $\mathsf A_1[(a,\mathbf v),(b,\mathbf w)] = 1$, so
\[
    (\mathsf A_1 \mathbf x)_{(a,\mathbf v)} = \sum_{b \neq a} q^2 g_b = q^2 \left ( \sum_{b \in \ZZ_{q+1}} g_b - g_a \right ) = -q^2 g_a .
\]
Hence $\mathsf A_1 \mathsf E_1 = - q^2\mathsf E_1$, and $\mathsf E_1$ does not split.

Consequently the image of $\mathsf E_3^\prime$ decomposes into its $q^2$- and $-q^2$-eigenspaces, of dimensions $m_3,m_4$, initially possibly zero.
Since $\tr \mathsf A_1 = 0$,
\[
    q^3 + q(-q^2) + \rank(\mathsf E_2)\cdot 0 + m_3 q^2 - m_4 q^2 = 0 ,
\]
so $m_3 = m_4$; and $m_3 + m_4 = \rank(\mathsf E_3^\prime) = q(q+1)(q^2-1)$ implies that $m_3 = m_4 = \frac{q(q+1)(q^2-1)}{2}$.

The dimensions are positive, so both eigenspaces occur.  
Their projections are the remaining primitive idempotents $\mathsf E_3,\mathsf E_4$, since the algebra has dimension five.
The remaining entries follow by linearity: since $\mathsf A_2 = \mathsf A_{\neq} - \mathsf A_1$, its eigenvalues are $\alpha - \theta$; and since $\sum_{j=0}^4 \mathsf A_j = J$, the eigenvalues of $\mathsf A_4$ on the non-principal idempotents are $-1 - \theta - (\alpha - \theta) - \beta = -1 - \alpha - \beta$. 
\end{proof}

As an immediate corollary, we find that the coherent closure of the graph $G_q$ is the above $4$-class association scheme.

\begin{corollary}
\label{cor:auxiliary-coherent-rank}
Let $q$ be a prime power.
Then
\[
    \mathcal W(G_q)
      =\operatorname{span}_{\RR}\{I,\mathsf A_1,\mathsf A_2,
                                    \mathsf A_3,\mathsf A_4\}.
\]
\end{corollary}

\begin{proof}
By Theorem~\ref{thm:main-scheme-and-eigenmatrix} the coherent closure $\mathcal W(G_q)$ is contained in $\operatorname{span}_{\RR}\{I,\mathsf A_1,\mathsf A_2,                                    \mathsf A_3,\mathsf A_4\}$.  
    To prove the reverse inclusion, put $H=\mathsf A_1^2-q^3I$.
By \eqref{eqn:A1sq}, $H=(q-1)\mathsf A_{\neq}+q^2\mathsf A_3$.
Since $\mathsf A_{\neq}$ and $\mathsf A_3$ have disjoint binary supports,
\[
    \mathsf A_3=
       \frac{H\circ H-(q-1)H}{q^2(q^2-q+1)},
    \qquad
    \mathsf A_{\neq}=\frac{H-q^2\mathsf A_3}{q-1}.
\]
Both matrices therefore belong to $\mathcal W(G_q)$.
The identities
\[
    \mathsf A_2=\mathsf A_{\neq}-\mathsf A_1,
    \qquad
    \mathsf A_4=J-I-\mathsf A_{\neq}-\mathsf A_3
\]
show that the remaining basis matrices also belong to the coherent closure.  
This proves equality.
\end{proof}

\subsection{Coherent types}
\label{sub:coherent-types}

For the infinite family, we identify the coherent closure by matching a lower bound with a containing algebra. 

\begin{lemma}
\label{lem:degree-block-rank}
Let $\Gamma$ be a biregular graph with degree partition $\{V_1,V_2\}$.
Let $G_i=\Gamma[V_i]$ for $i=1,2$.
Suppose $A(\Gamma) = \left [\begin{smallmatrix}A(G_1)&B\\B^\transpose&A(G_2)\end{smallmatrix}\right ]$, where $B\notin\{O,J\}$.
Then
\[
    \crk(\Gamma)\geqslant
    \crk(\Gamma[V_1])+\crk(\Gamma[V_2])+4.
\]
\end{lemma}

\begin{proof}
Write $A=A(\Gamma)$ and let $d_1\ne d_2$ be the two degrees.  
The diagonal matrix $D=A^2\circ I$ belongs to $\mathcal W(\Gamma)$, so the diagonal indicator matrices
\[
    F_1=\frac{D-d_2I}{d_1-d_2},\qquad F_2=I-F_1
\]
of $V_1,V_2$ belong to it.  
Restricted to $V_i\times V_i$, the corner $F_i\mathcal W(\Gamma)F_i$ is a coherent algebra containing $A(\Gamma[V_i])$.
Its dimension is therefore at least $\crk(\Gamma[V_i])$.  
Each off-diagonal corner contains the two linearly independent matrices $B,J_{|V_1|,|V_2|}$ or their transposes.
The four corners have disjoint supports, so their dimensions add, giving the stated bound.
\end{proof}

The restriction to the degree partition is important: an arbitrary equitable partition need not have its diagonal indicator matrices be in the coherent closure.

\begin{proposition}
\label{prop:family-coherent-type}
Let $q$ be a prime power.
Then the coherent closure $\mathcal W(\Gamma_q)$ has type $\left[\begin{smallmatrix}3&2\\2&5\end{smallmatrix}\right]$.
\end{proposition}

\begin{proof}
Let $m=q^3+1$, $v=mq^2$, $b=q^4(q+1)$.  
The adjacency matrix of $\Gamma_q$ is
\[
    A=\begin{bmatrix}J_v-I_m\otimes J_{q^2}&M_q\\M_q^\transpose&\mathsf A_1\end{bmatrix}.
\]
By Corollary~\ref{cor:main-three}, its point and block sets are its two degree classes.  
We exhibit a coherent algebra of dimension $12$ containing $A$, then apply Lemma~\ref{lem:degree-block-rank} to prove that it is the coherent closure.

Let
\[
\begin{gathered}
    \mathcal L=\operatorname{span}_{\RR}\{I_v,I_m\otimes J_{q^2},J_v\},\qquad
    \mathcal U=\operatorname{span}_{\RR}\{M_q,J_{v,b}\},\\
    \mathcal R=\operatorname{span}_{\RR}
       \{I_b,\mathsf A_1,\mathsf A_2,\mathsf A_3,\mathsf A_4\}.
\end{gathered}
\]
Here $\mathcal L$ is the rank-three Bose--Mesner algebra of the complete multipartite graph on the points, and $\mathcal R$ is the rank-five coherent closure of the auxiliary graph by Corollary~\ref{cor:auxiliary-coherent-rank}.

Since $G_q$ is an auxiliary graph of $\mathcal D_q$ by Lemma~\ref{lem:kernel-completion}, we can apply \eqref{eqn:aux-block} to obtain
\[
    M_q^\transpose M_q=q^4I+(q^3+q)J-\mathsf A_1^2.
\]
Using \eqref{eqn:A1sq} and $J=I+\mathsf A_{\neq}+\mathsf A_3+\mathsf A_4$, we obtain
\begin{equation}\label{eqn:MtM}
    M_q^\transpose M_q=(q^4+q)I+(q^3+1)\mathsf A_{\neq}
       +(q^3-q^2+q)\mathsf A_3+(q^3+q)\mathsf A_4.
\end{equation}

To check multiplication across the two degree classes, we use
\begin{align*}
    M_q\mathsf A_1&=q^2J,&
    M_q\mathsf A_2&=(q^4-q^2)J,\\
    M_q\mathsf A_3&=-qM_q+q^2J,&
    M_q\mathsf A_4&=(q-1)M_q+(q^3-q^2)J.
\end{align*}
The first identity follows from \eqref{eqn:aux-mixed}. 
Every row of $M_q$ has $q^3$ ones in each copy, so $M_q\mathsf A_{\neq}=q^4J$, giving the second.  
For the third, each of the $q$ used-class lines through a fixed vertex meets the row's $q$ selected free-class lines in $q$
points.  
Removing the vertex itself subtracts $q$ precisely when its incidence entry is one, giving $q^2-qM_q$.  
Since $I+\mathsf A_3+\mathsf A_4$ is the block-diagonal matrix with one all-ones block for each copy, and every row of $M_q$ has $q^3$ ones in each copy, we have
\[
    M_q(I+\mathsf A_3+\mathsf A_4)=q^3J.
\]
Using the third identity, we obtain
\[
    M_q\mathsf A_4
      =q^3J-M_q-M_q\mathsf A_3
      =(q-1)M_q+(q^3-q^2)J.
\]
These four identities show that $M_qR\in\mathcal U$ for every $R\in\mathcal R$.  
Since each such $R$ has constant column sum, we also have $JR\in\operatorname{span}_{\mathbb R}\{J\}$.
As $\mathcal U=\operatorname{span}_{\mathbb R}\{M_q,J\}$, it follows that $\mathcal U\mathcal R\subseteq\mathcal U$.

By Remark~\ref{rem:semireg}, we have $(I_m\otimes J_{q^2})M_q=qJ$.
Together with the constant row and column sums of $M_q$, one obtains $\mathcal L\mathcal U\subseteq\mathcal U$.
By \eqref{eqn:inc}, we have $\mathcal U\mathcal U^\transpose\subseteq\mathcal L$, and by \eqref{eqn:MtM}, we have $\mathcal U^\transpose\mathcal U\subseteq\mathcal R$.
Thus
\[
    \mathcal C=
    \left\{\begin{bmatrix}L&U\\V&R\end{bmatrix}:
       L\in\mathcal L,\ U\in\mathcal U,\
       V\in\mathcal U^\transpose,\ R\in\mathcal R\right\},
\]
with its four blocks chosen independently, is closed under ordinary multiplication and transpose.  
It is also closed under Schur multiplication: its disjoint binary basis consists of $I_v,I_m\otimes J_{q^2}-I_v,J_v-I_m\otimes J_{q^2}$ on the point side, $I_b,\mathsf A_1,\ldots,\mathsf A_4$ on the block side, and $M_q,J-M_q$ and their transposes between the two sides, each extended by zero outside its corner.  
These twelve matrices are nonzero, and $\mathcal C$ contains $I,J,A$.  
Hence $\mathcal W(\Gamma_q)\subseteq\mathcal C$ and $\dim\mathcal C=12$.

The induced point graph has coherent rank $3$, and the induced block graph has coherent rank $5$ by Corollary~\ref{cor:auxiliary-coherent-rank}.  
Also $M_q\notin\{O,J\}$ by its row sums.  
By Lemma~\ref{lem:degree-block-rank}, we have $\crk(\Gamma_q)\geqslant3+5+4=12$.
It follows that $\mathcal W(\Gamma_q)=\mathcal C$.
Its two diagonal basis relations are the identities on the point and block sets, giving the asserted fibres and type.
\end{proof}

We can run the Weisfeiler--Leman stabilisation algorithm~\cite{WeisfeilerLeman1968,Higman1987} to find the coherent closures of the graphs constructed in Sections~\ref{sec:row4-muzychuk} and \ref{sec:quadrangle}.
The coherent closure of the biregular graph constructed in Section~\ref{sec:quadrangle} has type $\left [ \begin{smallmatrix}3 & 2 \\2&3\end{smallmatrix} \right ]$.
On each fibre, its three basis relations are the identity, adjacency, and nonadjacency relations of the corresponding induced graph: the complete multipartite graph on the point set and the strongly regular auxiliary graph on the block set.
Thus both rank-three schemes are symmetric, and this coherent closure is a strongly regular design in the sense of Higman.
This answers \cite[Question~8.5]{GreavesYip2024} in the affirmative.
The coherent closure of the biregular graph constructed in Section~\ref{sec:row4-muzychuk} has type $\left [ \begin{smallmatrix}3 & 2 \\2&6\end{smallmatrix} \right ]$.

% Bibliography generated from the companion bib.bib using amsplain.
\providecommand{\bysame}{\leavevmode\hbox to3em{\hrulefill}\thinspace}
\providecommand{\MR}{\relax\ifhmode\unskip\space\fi MR }
% \MRhref is called by the amsart/book/proc definition of \MR.
\providecommand{\MRhref}[2]{%
  \href{http://www.ams.org/mathscinet-getitem?mr=#1}{#2}
}
\providecommand{\href}[2]{#2}

\end{document}